\documentclass[11pt, a4paper]{article}
\usepackage{amsmath, amsfonts}
\usepackage[usenames]{color}
\usepackage[english]{babel}

\usepackage[table, dvipsnames]{xcolor}
\usepackage{array}
\newcolumntype{P}[1]{>{\centering\arraybackslash}p{#1}}
\newcolumntype{M}[1]{>{\centering\arraybackslash}m{#1}}
\usepackage{multirow}
\usepackage{tabularx}
\usepackage[all]{xy}
\usepackage{tikz-cd}
\usepackage{tikz}
\usepackage{comment}
\usepackage{float}
\usepackage{xcolor}
\usepackage{graphicx}
\usepackage{geometry}
\usepackage{mathtools}
\usepackage{mathrsfs}
\usepackage{tcolorbox}
\usepackage[table]{xcolor}
\usepackage{hyperref} 
\hypersetup{bookmarks = true, 
	colorlinks = true, 
	linkcolor = blue, 
	citecolor = blue, 
}
\usepackage{tcolorbox}
\definecolor{mygray}{rgb}{0.98,0.95,0.75}
\newtcolorbox{mybox}{
	arc=0pt,
	boxrule=0pt,
	colback=mygray,
	width=1\textwidth,   
	colupper=black,
	fontupper=\normalsize
}

\newcommand{\xdownarrow}[1]{%
	{\left\downarrow\vbox to #1{}\right.\kern-\nulldelimiterspace}
}

\newcommand{\xuparrow}[1]{%
	{\left\uparrow\vbox to #1{}\right.\kern-\nulldelimiterspace}
}

\usepackage{graphicx}
\usepackage{amssymb}
\usepackage{amsmath}

\newtheorem{theorem}{Theorem}[section]

\newtheorem{corollary}[theorem]{Corollary}

\newtheorem{definition}[theorem]{Definition}
\newtheorem{example}[theorem]{Example}

\newtheorem{lemma}[theorem]{Lemma}

\newtheorem{proposition}[theorem]{Proposition}
\newtheorem{remark}[theorem]{Remark}

\usetikzlibrary{arrows.meta}
\usetikzlibrary{arrows}
\tikzset{
  doubleimplies/.style={-Implies, double, thick}
}

\newenvironment{proof}[1][Proof]{\noindent\textbf{#1.} }{\ \rule{0.5em}{0.5em}}

\title{\textbf{Transverse slices, Ruas' conjecture, and Zariski's multiplicity conjecture for quasihomogeneous surfaces}}
\author{ \ \ \ \\{Silva, O. N. $ \ \ $ and  $ \ \ $ Silva Jr, M. M. $ \ \ $}}
\date{}

\begin{document}
	
	\maketitle
	
	\begin{abstract}
		In this work, we consider a finitely determined, quasihomogeneous, corank 1 map germ $f$ from $(\mathbb{C}^2,0)$ to $(\mathbb{C}^3,0)$. We introduce the concept of the \textit{$\mu_{\mathbf{m},\mathbf{k}}$-minimal transverse slice of $f$}. Since such a slice is a plane curve, it admits a topological normal form, which we describe explicitly. Assuming the $\mu_{\mathbf{m},\mathbf{k}}$-minimal transverse slice hypothesis, we provide a proof for the equivalence between topological triviality and Whitney equisingularity in Ruas’ conjecture within this setting. We also provide a counterexample which shows that Whitney equingularity does not imply bi-Lipschitz equisingularity, given an answer to a question by Ruas. Moreover, we show that every topologically trivial $1$-parameter unfolding of $f=(f_1,f_2,f_3)$ (not necessarily with $\mu_{\mathbf{m},\mathbf{k}}$-minimal transverse slice) is of non-negative degree; that is, any additional term $\alpha$ in the deformation of $f_i$ has weighted degree not smaller than that of $f_i$. As a consequence, we provide a proof of Zariski’s multiplicity conjecture for 1-parameter families of such germs.
	\end{abstract}
	
	

	\section{Introduction}
	
	$ \ \ \ \ $ In the 1970s, Oscar Zariski established a general theory of equisingularity for hypersurfaces over fields of characteristic zero (see \cite{Zariski1979} and \cite{Zariski}). Since then, the study of several notions of equisingularity has remained central in singularity theory, with topological triviality and Whitney equisingularity serving as its classical criteria. These notions connect equisingularity to invariants such as the Milnor number and to regularity conditions on stratifications (see for instance \cite{Briancon} and \cite{whitney65}). Several fundamental problems, such as Zariski’s multiplicity conjecture for families of hypersurfaces \cite{Conjzari}, Ruas’ conjecture for families of finitely determined map germs from $(\mathbb{C}^2,0)$ to $(\mathbb{C}^3,0)$ \cite{Conjruas}, and Houston’s conjecture for families of map germs from $(\mathbb{C}^n,0)$ to $(\mathbb{C}^{n+1},0)$ \cite[Remark 7.3]{Houston}, continue to guide the development of singularity theory, particularly in the study of equisingularity and invariants of map germs. Even though Ruas’s conjecture has received an answer in a weaker sense \cite{Ruas} and Houston’s conjecture has been settled \cite{roberto}, Zariski’s multiplicity conjecture remains open in its original formulation.

	In the case of surfaces in $\mathbb{C}^3$, the technique of transversal slices is a powerful tool that simplifies the study of equisingularity by reducing the problem to the analysis of a corresponding family of plane curves. In this work, we study finitely determined, quasihomogeneous, corank 1 map germs from $(\mathbb{C}^2,0)$ to $(\mathbb{C}^3,0)$ and the topology of their transverse slices, as introduced by Marar and Nuño-Ballesteros in \cite{MararJuan}. Our main focus is on investigating conditions under which Ruas’ conjecture and Zariski’s multiplicity conjecture hold. To this end, we make use of several techniques, among which we introduce a minimality condition on the transverse slice of $f$, which we call the \textit{$\mu_{\textbf{m},\textbf{k}}$-minimal transverse slice}.
	 
Before discussing the problems considered in this work and the methods used to solve them, we provide a brief motivation for the notion of a $\mu_{\mathbf{m},\mathbf{k}}$-minimal plane curve germ. In this setting, $r$ denotes the number of branches of the curve, while $\mathbf{m}=(m_1,\dots,m_r)$ and $\mathbf{k}=(k_1,\dots,k_r)$ are multi-indices naturally associated with the embedded topology of these branches (see Definition \ref{min}).
	 	 
	Consider a germ of a reduced, equidimensional complex surface $(X,0)$ together with a flat analytic map $p:(X,0)\to(\mathbb{C},0)$, and let $p:X\to T$ be a representative, where $T$ is an open neighborhood of $0$ in $\mathbb{C}$. The surface $X$ can be viewed as a flat $1$-parameter deformation of the curve $X_0 := p^{-1}(0)$. Suppose that the singular locus of $X$ is smooth of dimension one, and that there exists a section $\sigma:T\to X$, such that the image of $\sigma$ is smooth and each fiber $X_t:=p^{-1}(t)$ has a unique singular point at $\sigma(t)$ (see Figure \ref{Figure3}). A classical result asserts that $X$ is topologically trivial if and only if the Milnor number $\mu(X_t,\sigma(t))$ of its fibers remains constant along $\sigma(T)$ (see \cite[Th. $5.2.2$]{Buchweitz}).

  \begin{figure}[H]
        \centering
\includegraphics[scale=0.5]{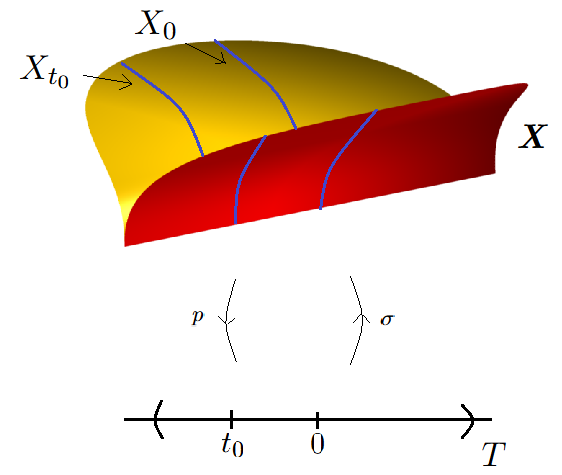} 
        \caption{The notion of a family of curves.}
        \label{Figure3}
    \end{figure}	
		
	Furthermore, $X$ is Whitney equisingular precisely when it is topologically trivial and the multiplicity $m(X_t,\sigma(t))$ of $X_t$ at $\sigma(t)$ remains constant for all sufficiently small $t \in T$ (see \cite[Th. III.3]{Briancon}).
	
	If $(X_0,0)$ and $(X_t,\sigma(t))$ are topologically equivalent germs of reduced plane curves, then a classical result of Zariski \cite{Zariskii} states that their multiplicities coincide, i.e., $m(X_0,0)=m(X_t,\sigma(t))$. Furthermore, Neumann and Pichon \cite{Pichon} show that $(X_0,0)$ and $(X_t, \sigma(t))$ have the same Lipschitz geometry; that is, there exists a homeomorphism of germs $h: (\mathbb{C}^2,0)\to(\mathbb{C}^2,0)$ with $h(X_0)=X_t$, which is bi-Lipschitz with respect to the outer metric. Consequently, the Milnor number of the fibers $X_t := p^{-1}(t)$ at $\sigma(t)$ is the unique invariant that completely characterizes Whitney equisingularity (equivalently, bi-Lipschitz equisingularity) for a family of reduced plane curves $p:X\rightarrow T$.
	
Let us now consider the case where $(S,0)$ is a surface in $(\mathbb{C}^3,0)$. 
Suppose that $(S,0)$ admits a parametrization by a finitely determined map germ. 
Equivalently, $(S,0)$ is the image of a finitely determined map germ 
$f:(\mathbb{C}^2,0)\to(\mathbb{C}^3,0)$. 
In this context, the double point curve $D(f)$ is fundamental for understanding 
the topology of the image of $f$ (see Section \ref{d2f}). 
Now, consider a 1-parameter origin-preserving unfolding $F=(f_t,t)$ of $f$. 
Clearly, the unfolding $F$ induces a deformation of $S$, and we set $S_t:=f_t(\mathbb{C}^2)$. 
Ruas (1994) conjectured that the topological triviality of the family $S_t$ 
is completely characterized by the topological triviality of the family of 
plane curves $D(f_t)$.

In other words, controlling the topological triviality of the family of surfaces $S_t$ 
amounts to controlling the topological triviality of the family of plane curves $D(f_t)$. 
As we have seen above, for plane curves, topological triviality, Whitney equisingularity, and bi-Lipschitz equisingularity are all characterized by the constancy of the Milnor number of the fibers along the parameter space. 
This naturally raises the question of whether the Whitney equisingularity (respect. bi-Lipschitz equisingularity) of $S_t$ 
can be controlled by the constancy of the Milnor number of the double point curve $D(f_t)$ at the origin. 

When working with a parametrization of a surface, the notions of topological triviality, Whitney equisingularity, and bi-Lipschitz equisingularity can likewise be defined for unfoldings of such maps (see Definition \ref{whit}). One formulation of Ruas’ conjecture for unfoldings of a map germ is the following (see for instance \cite[Th. 1.2]{Bedregal}):

\begin{mybox}
		\textbf{Ruas' conjecture (1994):}  Let $f:(\mathbb{C}^2,0)\rightarrow(\mathbb{C}^3,0)$ be a finitely determined map germ. Consider a 1-parameter unfolding $F:(\mathbb{C}^2\times\mathbb{C},0)\rightarrow(\mathbb{C}^3\times\mathbb{C},0)$ of $f.$ The following statements are equivalent:\\
		
\noindent $\bullet$ ($\mu$): $\mu(D(f_t),0)$ is constant along the parameter space.\\
\noindent $\bullet$ $(Top)$: $F$ is topologically trivial.\\
\noindent $\bullet$ $(W)$: $F$ is Whitney equisingular.\\
\noindent $\bullet$ $(Lip)$: $F$ is bi-Lipschitz equisingular.

	\end{mybox}

In 2006, Callejas-Bedregal, Houston, and Ruas obtained the equivalence between statements ($\mu$) and $(Top)$ in Ruas’ conjecture in an unpublished manuscript \cite[Th. 6.2]{Bedregal}. This result was later confirmed independently by Fernández de Bobadilla and Pe Pereira in 2008 \cite[Cor. 40]{lev}, and again by Nuño-Ballesteros and Tomazella in 2012 \cite[Th. 4.2]{Nuno2}.
   
   Nevertheless, the equivalence between $(Top)$ and $(W)$ remained unresolved for almost $25$ years, while the equivalence between $(W)$ and $(Lip)$ still constitutes an open problem in the field. According to Thom’s second isotopy lemma for analytic maps (see \cite[Th. 5.2]{Gibson}), we obtain that $(W)$ implies $(Top)$. Unsuccessful attempts to prove this conjecture in its entirety or at least partially were made by Callejas-Bedregal, Houston, and Ruas \cite{Bedregal} in 2006, Ruas \cite{Tentruas} in 2013, and indirectly by Marar and Nuño-Ballesteros in \cite{MararJuan}, where they tried to study the particular case of corank $1$ with the development of the invariant $J(f).$ 
    
    Finally, in 2016, Ruas and the first author \cite{Ruas} (see also \cite{[10]}) exhibited counterexamples to the question, showing that the equivalence between $(Top)$ and $(W)$ is not true in general. It follows by Definition \ref{defwe} that $(Lip)$ implies in $(Top)$. Diagram \ref{diagram} summarizes the current state of this problem, including known results and open questions.

\begin{figure}[h!]
\centering    
\begin{tikzpicture}
[
  node distance=35mm and 45mm,
  box/.style={
    rounded corners=3pt, 
    draw, 
    align=center, 
    minimum width=45mm, 
    minimum height=10mm, 
    inner sep=4pt,
    fill=mygray
  },
  imply/.style={->, thick, shorten >=4pt, shorten <=4pt},
  equiv/.style={<->, thick, shorten >=4pt, shorten <=4pt},
  notimply/.style={->, thick, dashed, shorten >=4pt, shorten <=4pt},
  unknown/.style={->, thick, dotted, shorten >=4pt, shorten <=4pt},
  diagimply/.style={->, thick, shorten >=8pt, shorten <=8pt},
  diagnot/.style={->, thick, dashed, shorten >=8pt, shorten <=8pt},
  diagunk/.style={->, thick, dotted, shorten >=8pt, shorten <=8pt}
]

\node[box] (A) at (0,2.45) {Constancy of $\mu(D(f_t),0)$};
\node[box] (B) at (0,0)   {Topological Triviality};
\node[box] (C) at (-4,-2.45)    {Whitney equisingularity};
\node[box] (D) at (4,-2.45)    {bi-Lipschitz equisingularity};

\node(E) at (0.35,1.23) {\cite{Bedregal}};

\node(F) at (-1.5,-1.4) {$|$};

\draw[Implies-Implies, double, thick, shorten >=10pt, shorten <=10pt] (A) -- (B) node[midway, sloped, above]{};

\draw[-{Implies[length=3pt]}, double, thick, shorten >=10pt, shorten <=10pt] (C) -- (B)  node[midway, sloped, above] {\cite{Gibson}};

\draw[-Implies, double, thick, shorten >=70pt, shorten <=10pt] ([yshift=-0.001cm]B.south) -- ([yshift=-0.001cm]C.south)
      node[midway, font=\small] {} node[pos=0.27, sloped, below] {\cite{Ruas}};  

\draw[Implies-Implies, double, thick, shorten >=5pt, shorten <=5pt] (D) -- (C)  node[pos=0.5, sloped, above] {$?$}; 

\draw[Implies-, double, thick, shorten >=10pt, shorten <=10pt] (B) -- (D)  node[midway, sloped, above, font=\small] {Def.};

\draw[-Implies, double, thick, shorten >=70pt, shorten <=10pt] ([yshift=-0.001cm]B.south) -- ([yshift=-0.001cm]D.south)
      node[midway, font=\small] {} node[pos=0.27, sloped, below] {$?$};

\end{tikzpicture}
\caption{Implication diagram for Ruas’ conjecture.}
\label{diagram}
\end{figure}
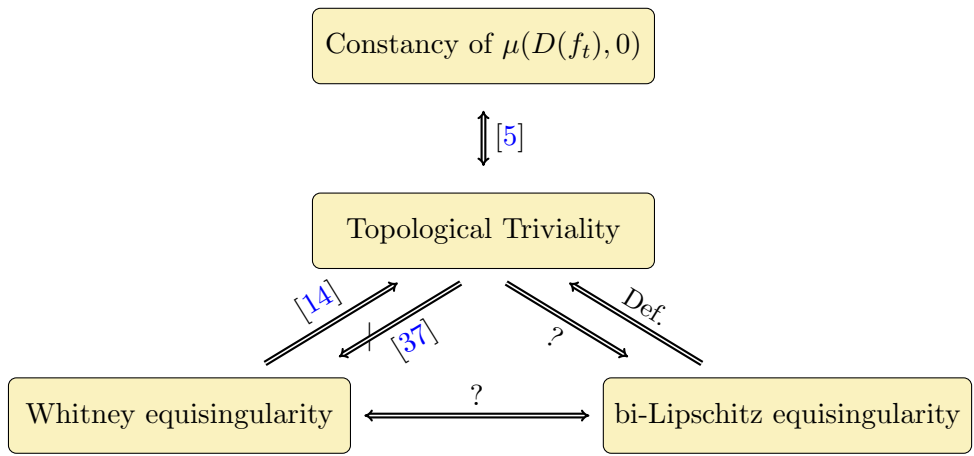
     
In particular, one of the counterexamples given in \cite{Ruas} is the family $f_t:(\mathbb{C}^2,0)\rightarrow (\mathbb{C}^3,0)$ defined by $$f_t(x,y)=(x,y^4,x^5y-5x^3y^3+4xy^5+y^6+ty^7).$$      

Throughout this work, we shall refer to the surface $R := f_0(\mathbb{C}^2)$ as Ruas’s surface (see Figure \ref{Figure2}(a)), in memory of Professor Maria Aparecida Soares Ruas, lovingly remembered by everyone as ``Cidinha''. To the best of our knowledge, \cite[Example 5.5]{Ruas} is the first example reported in the literature of a surface parametrized by a finitely determined map germ that admits a topologically trivial deformation that fails to be Whitney equisingular. 

     \begin{figure}[H]
        \centering
\includegraphics[scale=0.54]{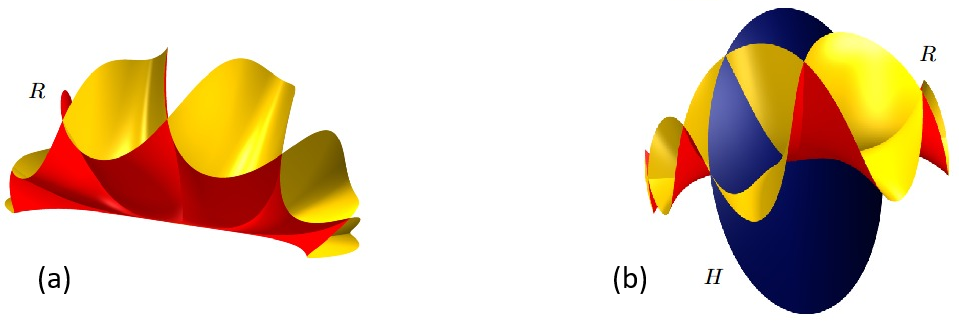} 
        \caption{Transverse slice of Ruas’s surface.}
        \label{Figure2}
    \end{figure}
    
    In this family, the transverse slice of $f_t$ (see Definition \ref{slice}) is an irreducible plane curve with three characteristic exponents: $4,6,$ and $9$ when $t=0$, and $4,6,$ and $7$ when $t\ne0$ (see \cite[Prop. 3.10]{slice}).  Figure \ref{Figure2}(b) illustrates the transversal slice of $f_0$ which parametrizes Ruas's surface. 
    
     An irreducible plane curve with characteristic exponents $4,6,$ and $7$ attains the minimal Milnor number among all irreducible analytic plane curve germs of multiplicity $4$ with three characteristic exponents. Ruas's surface provides the motivation for the introduction of the notion of a $\mu_{\textbf{m},\textbf{k}}$-minimal plane curve.
    
    This type of curve is defined as a plane curve with $r$ branches that attains the minimum Milnor number under a certain fixed condition (see Definition \ref{min}). An interesting fact is that the transverse slice of each counterexample in \cite[Table 4.1]{[10]} is not $\mu_{m,k}$-minimal. So, a very natural question is the following one:

    \begin{mybox}
		\textbf{Question 1:} Let $f:(\mathbb{C}^2,0)\rightarrow(\mathbb{C}^3,0)$ be a finitely determined map germ, and let $F=(f_t,t)$ be a topologically trivial $1$-parameter unfolding of $f$.\\
		
		 If the transverse slice $\gamma$ of $f$ is $\mu_{m,k}$-minimal, does it follow that $F$ is Whitney equisingular?
	\end{mybox}

In this work, under two additional hypotheses, we present a positive answer to Question $1$, which forms our first main result. More precisely, we prove the following:

    \begin{theorem}\label{1.1}
        Let $f:(\mathbb{C}^2,0)\rightarrow(\mathbb{C}^3,0)$ be a finitely determined, quasihomogeneous, corank 1 map germ. Suppose that $F=(f_t,t)$ is a topologically trivial 1-parameter unfolding of $f.$ If the transverse slice $\gamma$ of $f$ is $\mu_{m,k}$-minimal, then $F$ is Whitney equisingular.
    \end{theorem}
    
    As a consequence of Theorem \ref{1.1}, we identify a class of map germs for which the equivalence between $(Top)$ and $(W)$ in Ruas’ conjecture holds. In particular, Hironaka’s result (see \cite[Cor. 6.2]{Hironaka}), which states that Whitney regular stratifications imply equimultiplicity along the strata, allows us to conclude that Zariski’s multiplicity conjecture \cite{Conjzari}, in its family version, also holds for this class of hypersurfaces.
    
On the other hand, we present a counterexample showing that the equivalence between $(W)$ and $(Lip)$ in Ruas's conjecture does not hold in general. We also present an infinite list of counterexamples to the implication $(Top) \Rightarrow (W)$ in Ruas’ conjecture (see Section \ref{newex}). These families are topologically trivial but fail to satisfy Whitney’s $(b)$-regularity condition, thereby providing concrete situations where topological triviality does not ensure $(b)$-regularity. This constitutes our second significant contribution (see Propositions \ref{whitbil}, \ref{contraexemplo}, and Corollary \ref{corcontra}).

According to Trotman \cite[p. 18]{Trotman} (see also \cite[p. 657]{Trot}), a question originally raised by Thom, it remains unknown whether a topologically trivial complex analytic family necessarily satisfies Whitney’s $(a)$-regularity condition. Our new examples may shed light on this open problem, since determining whether they fulfill $(a)$-regularity could represent a significant step toward resolving Thom’s long-standing question.

\begin{mybox}
		\textbf{Thom's Question:} Does topological triviality along the strata imply Whitney $(a)$-regularity for complex analytic stratifications?
	\end{mybox}

The new counterexamples to Ruas’ conjecture suggest that the original formulation requires refinement. In what follows, we present a new statement, hereafter referred to as \textit{Ruas’ Conjecture Revisited}.
 For the statement below, equisingular for $F$ means either topological trivial, Whitney equisingular, or bi-Lipschitz equisingular.

 \begin{mybox}
		\textbf{Ruas’ Conjecture Revisited:} Let $f:(\mathbb{C}^2,0)\rightarrow(\mathbb{C}^3,0)$ be a finitely determined map germ. Consider a 1-parameter unfolding $F:(\mathbb{C}^2\times\mathbb{C},0)\rightarrow(\mathbb{C}^3\times\mathbb{C},0)$ of $f.$\\
		
Associated to $F=(f_t,t)$, there exists a family of plane curve germs $(X_t,0)$ in the source of $F$ for which equisingularity of $F$ is equivalent to the constancy of the Milnor number of $X_t$ along the parameter space.		
	\end{mybox}

Note that the equivalence between $(\mu)$ and $(Top)$ in the original formulation of Ruas’ conjecture is a particular case of the revisited version. In this setting, the family of curves $X_t$ coincides precisely with the family of double point curves $D(f_t)$. According to the new reformulation, $F$ is topologically trivial if, and only if the Milnor number of $D(f_t)$ is constant along the parameter space.

In this work, we provide a proof to Ruas' conjecture revisited for the case in which $F$ is Whitney equisingular. More precisely, we define a new plane curve in the source, namely, $W(f):=D(f) \cup f^{-1}(\gamma)\subset \mathbb{C}^2$ and we prove our third main result:  

\begin{theorem}\label{main result 3} Let $f:(\mathbb{C}^2,0)\rightarrow(\mathbb{C}^3,0)$ be a finitely determined map germ and let $F:(\mathbb{C}^2 \times \mathbb{C},0)\rightarrow (\mathbb{C}^3 \times \mathbb{C},0)$, $F=(f_t,t)$, be an unfolding of $f$. Set $(W(f_t),0):=(D(f_t)\cup f_t^{-1}(\gamma_t),0)$, where $\gamma_t$ is the transversal slice of $f_t(\mathbb{C}^2)$. Then

\begin{center}
$F$ is Whitney equisingular $ \ \ \ $ $\Longleftrightarrow$ $ \ \ \ $ $\mu(W(f_t),0)$ is constant.
\end{center}
\end{theorem}

The next diagram is an update of the Diagram \ref{diagram} with the results obtained in the work, where the existence of the family of plane curves $L_t$ is conjectured in Section \ref{sec8}.

\begin{figure}[h!]
\centering    
\begin{tikzpicture}
[
  node distance=35mm and 45mm,
  box/.style={
    rounded corners=3pt, 
    draw, 
    align=center, 
    minimum width=45mm, 
    minimum height=10mm, 
    inner sep=4pt,
    fill=mygray
  },
  imply/.style={->, thick, shorten >=4pt, shorten <=4pt},
  equiv/.style={<->, thick, shorten >=4pt, shorten <=4pt},
  notimply/.style={->, thick, dashed, shorten >=4pt, shorten <=4pt},
  unknown/.style={->, thick, dotted, shorten >=4pt, shorten <=4pt},
  diagimply/.style={->, thick, shorten >=8pt, shorten <=8pt},
  diagnot/.style={->, thick, dashed, shorten >=8pt, shorten <=8pt},
  diagunk/.style={->, thick, dotted, shorten >=8pt, shorten <=8pt}
]

\node[box] (A) at (0,2.45) {Constancy of $\mu(D(f_t),0)$};
\node[box] (B) at (0,0)   {Top.\ Triviality};
\node[box] (C) at (-4,-2.45)    {Whitney eq.};
\node[box] (D) at (4,-2.45)    {bi-Lipschitz eq.};
\node[box] (G) at (-4,-4.9)    {Constancy of $\mu(W_t,0)$};
\node[box] (H) at (4,-4.9)    {Constancy of $\mu(L_t,0)$};

\node(E) at (-3.25,-3.68) {\small{Th. \ref{w_t}}};
\node(J) at (4.25,-3.68) {$?$};
\node(I) at (0.35,1.23) {\cite{Bedregal}};

\node(F) at (-1.5,-1.4) {$|$};
\node(M) at (1.1,-1.2) {$|$};
\node(K) at (0,-2.9) {$/$};
\node(L) at (0,-3.3) {\small{Prop. \ref{whitbil}}};
\node(L) at (0,-2.2) {\small{Th. \ref{ab} (corank 1)}};

\draw[Implies-Implies, double, thick, shorten >=10pt, shorten <=10pt] (A) -- (B) node[midway, sloped, above]{};
\draw[Implies-Implies, double, thick, shorten >=10pt, shorten <=10pt] (C) -- (G) node[midway, sloped, above]{};

\draw[Implies-Implies, double, thick, shorten >=10pt, shorten <=10pt] (D) -- (H) node[midway, sloped, above]{};

\draw[-{Implies[length=3pt]}, double, thick, shorten >=10pt, shorten <=10pt] (C) -- (B)  node[midway, sloped, above] {\cite{Gibson}};

\draw[-Implies, double, thick, shorten >=70pt, shorten <=10pt] ([yshift=-0.001cm]B.south) -- ([yshift=-0.001cm]C.south)
      node[midway, font=\small] {} node[pos=0.27, sloped, below] {\cite{Ruas}};  

\draw[Implies-, double, thick, shorten >=70pt, shorten <=70pt] ([xshift=0.001cm]D.south) -- ([xshift=0.001cm]C.south)
      node[midway, font=\small] {} node[pos=0.27, sloped, below] {};      

\draw[-Implies, double, thick, shorten >=5pt, shorten <=5pt] (D) -- (C)  node[pos=0.5, sloped, above] {};

\draw[Implies-, double, thick, shorten >=10pt, shorten <=10pt] (B) -- (D)  node[midway, sloped, above, font=\small] {Def.};

\draw[-Implies, double, thick, shorten >=70pt, shorten <=10pt] ([yshift=-0.001cm]B.south) -- ([yshift=-0.001cm]D.south)
      node[midway, font=\small] {} node[pos=0.27, sloped, below] {\small{Prop. \ref{whitbil}}};

\end{tikzpicture}
\caption{Update implication diagram for Ruas’ conjecture.}
\label{diagram2}
\end{figure}
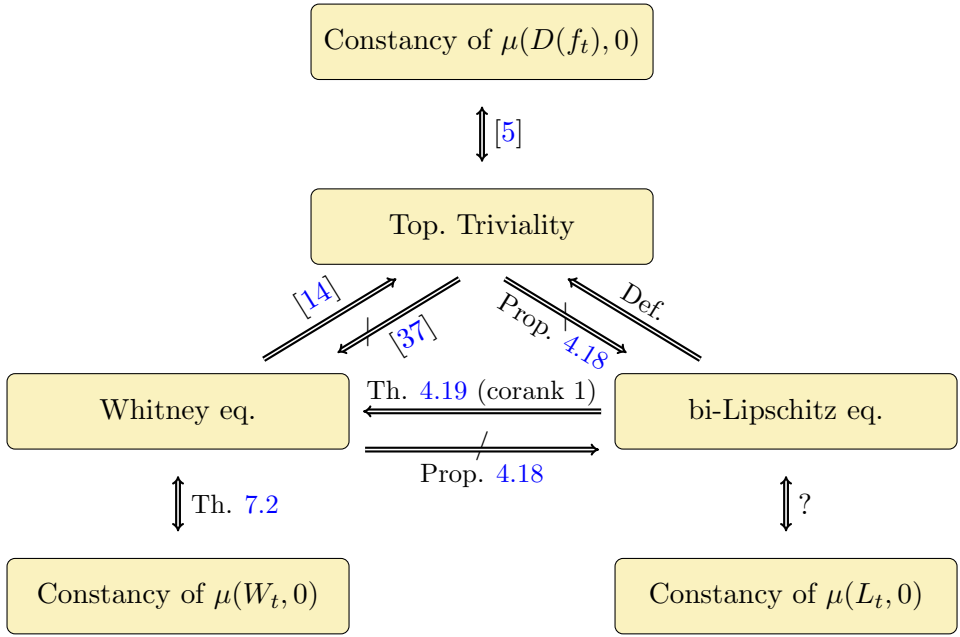
    
    About the Zariski’s multiplicity conjecture, this problem was proposed in 1971 by Oscar Zariski \cite{Conjzari}, during his address upon stepping down as president of the American Mathematical Society.  In this work, we consider the version for families of this conjecture. 
    
    More precisely, let $f:(\mathbb{C}^n,0)\rightarrow(\mathbb{C},0)$ be a reduced holomorphic function germ, $V(f)=f^{-1}(0)$ the corresponding hypersurface germ in $\mathbb{C}^n$. Let $F=(f_t,t)$ be a 1-parameter unfolding of $f,$ that is, $f_t$ is a reduced holomorphic function germ, $f_0=f$ and $f_t(0)=0,$ for all sufficiently small $t.$ Geometrically, the multiplicity of $V(f_t)=f_t^{-1}(0)$ at the origin is the number of points in the intersection of $V(f_t)$ with a generic line passing close to the origin but not through the origin. Denote by $m(V(f_t),0)$ the multiplicity of $V(f_t)$ at the origin. Furthermore, one says that the $F$ is topologically trivial if there exists a germ of homeomorphism $h_t:(\mathbb{C}^n,0)\rightarrow(\mathbb{C}^n,0)$ such that $h_t(V(f_t),0)=(V(f),0)$ for all sufficiently small $t.$ Now, we can present the problem:

    \begin{mybox}
		\textbf{Zariski's multiplicity conjecture for families:}  Let $f:(\mathbb{C}^n,0)\rightarrow(\mathbb{C},0)$ be a germ of a holomorphic function, and let $F=(f_t,t)$ be a 1-parameter unfolding of $f.$\\
		
		 If $F$ is topologically trivial, then $m(V(f_t),0)$ is constant.
	\end{mybox}

    As mentioned earlier, Zariski’s multiplicity conjecture remains an open problem, and numerous contributions addressing this problem are available in the literature. Here, we cite some special cases for the conjecture in its most general context, that is, in the case of a pair $(X,0)$ and $(Y,0).$ For example, the first contribution was made by Zariski himself \cite{Zariskii}, in the case $n=2.$

    In 1973, due to the works of A'Campo and Lê published that year, it became possible to obtain a solution for the case when $m(X,0)=1$ (see \cite[Th. 3]{ACampo} and \cite[Prop.]{Le}). In particular, an important consequence of this case is that if $(X,0)$ and $(Y,0)$ are topologically equivalent and $(X,0)$ has an isolated singularity, then $(Y,0)$ also has an isolated singularity. Later, in 1980, Navarro Aznar \cite{Navarro} proved the conjecture for the case in which $m(X,0)=2$ and $n=3.$

    Among recent advances, it is worth highlighting the important result obtained by Fernández de Bobadilla and Pełka \cite{Pelka}, which ensures that if a family $X_t$ of hypersurfaces with isolated singularity is $\mu$-constant, then $X_t$ is equimultiple. For an overview of the conjecture, the reader may also consult the work of Eyral \cite{Eyral}.

    In Section \ref{sec4}, we study the equimultiplicity of a topologically trivial family of surfaces (not necessarily with a $\mu_{m,k}$-minimal transverse slice). A common strategy is to impose an additional assumption on the map germ. For example, one may consider quasihomogeneous map germs. In this direction, in 1988, Saeki (see \cite{Saeki}) and independently Xu and Yau (see \cite{Xu} and \cite{Yau}) presented a proof for the conjecture for a pair of quasihomogeneous surfaces with isolated singularity in the case where $n=3.$ Furthermore, the first author (see \cite{OtoZari}) presented a proof of the conjecture for a pair of quasihomogeneous parametrized surfaces with non-isolated singularity also in the case where $n=3.$ Motivated by this result, we now turn our attention to families of such surfaces. More precisely, we show that if $f:(\mathbb{C}^2,0)\rightarrow(\mathbb{C}^3,0)$  is a finitely determined, quasihomogeneous, corank 1 map germ, then every topologically trivial 1-parameter unfolding $F=(f_t,t)$ of $f$ is equimultiple (see Theorem \ref{zariski}).
    
    Our strategy for proving Theorem \ref{zariski} is to use important results given by Damon \cite{Damon} and  \cite{Varchenko}. In \cite{Damon}, Damon showed that if $f=(f_1,\dots,f_p)$ is a finitely determined quasihomogeneous map germ from $(\mathbb{C}^n,0)$ to $(\mathbb{C}^p,0),$ with $p>n,$ and $F$ is a 1-parameter unfolding of $f$ such that $F$ is of non-negative degree (i.e., for any term $\alpha$ in the deformation of $f_i$, the weighted degree of $\alpha$ is not smaller than the weighted degree of $f_i$), then $F$ is topologically trivial. In \cite{Varchenko}, Varchenko showed that for a 1-parameter unfolding $F=(f_t,t)$ of a quasihomogeneous map germ $f$ from $(\mathbb{C}^n,0)$ to $(\mathbb{C},0)$ with isolated singularity, such that 
\begin{equation*}
f_t(x)=f(x)+\sum_{j=1}^k\sigma_j(t)\alpha_j(x)    
\end{equation*}
\noindent with $\sigma_j(0)=0,$ then the Milnor number of $f_t,$ $\mu(f_t),$ is constant along the parameter space if and only if the weighted degree of $\alpha_j$ is not smaller than the weighted degree of $f,$ for all $j=1,\dots,k.$

    When $p>n+1$, \cite[Example 7.2.1]{Buchweitz} shows that the converse of Damon's result does not hold, in general. Inspired by Varchenko's result, one might ask whether the converse of Damon’s result holds for $p=n+1$. This motivates the following question:

    \begin{mybox}
		\textbf{Question 2:} Let $f:(\mathbb{C}^n,0)\rightarrow(\mathbb{C}^{n+1},0)$ be a finitely determined, quasihomogeneous map germ. Is it true that any topologically trivial 1-parameter unfolding $F=(f_t,t)$ of $f$ must be of non-negative degree?
	\end{mybox}

    In this work, we provide a positive answer to Question 2 in the case of corank 1 map germs from $(\mathbb{C}^2,0)$ to $(\mathbb{C}^3,0)$ (see Lemma \ref{Teo2}), which is considered our fourth main result. We use it to present a proof of fifth result (see Theorem \ref{zariski}).

      We return our attention to the study of $\mu_{\textbf{m},\textbf{k}}$-minimal plane curve germs. It is well known in the literature that, for plane curves, the characteristic exponents of their branches and its intersection multiplicities determine and are determined by the (embedded) topology of the curve (see, for instance, \cite{Greuel}). In this sense, the normal forms of the plane curves we address here are related to the topology of the curve. That is, if $(Y,0)$ is a $\mu_{\textbf{m},\textbf{k}}$-minimal plane curve, then it is topologically equivalent to another curve $(X,0)$ parametrized by that normal form. We use Puiseux parametrizations to describe these normal forms. In particular, when dealing with a $\mu_{\textbf{m},\textbf{k}}$-minimal plane curve, can we explicitly construct a normal form? In this work, we present a positive answer to this question (see Sections \ref{normalform1} and \ref{normalform}).
	
	\section{Preliminaries}\label{sec1}
	
	$ \ \ \ \ $  Throughout this work, we assume that $f:(\mathbb{C}^2,0)\rightarrow(\mathbb{C}^3,0)$ is a germ of a finite holomorphic map. Moreover, whenever we refer to a “plane curve germ”, we mean an “analytic plane curve germ”, unless stated otherwise. Finally, we use standard notation from Singularity Theory, which the reader can find in \cite{[7]} and \cite{Greuel}.

    \subsection{Plane curve germs}

    $ \ \ \ \ $ As a starting point, we present a well-known result from the literature. In this work, as in \cite{Greuel}, $\mathbb{C}\left\{x_1,\dots,x_n\right\}$ denote the local ring of absolutely convergent power series in the indeterminates $x_1,\dots,x_n$ in a neighborhood of the origin in $\mathbb{C}^n.$ 
    
    To begin with, consider $f\in\mathfrak{m}\subset\mathbb{C}\left\{ x,y\right\}$ be irreducible and $y$-generic of order $m.$ There exists $y(u)\in\langle u\rangle\cdot\mathbb{C}\langle u\rangle$ such that $$f(u^m,y(u))=0.$$ Moreover, the map $u\mapsto (u^m,y(u))$ is a parametrization of $f,$ so-called a Puiseux parametrization. For more details of this parametrizations, see Theorem 3.3 in \cite{Greuel}, page 163.

    With this concept, we are able to define the concept of characteristic exponents through the following construction, as done in \cite{Greuel}. Consider an irreducible plane curve germ $(X,0)\subset(\mathbb{C}^2,0)$ with an isolated singularity and multiplicity $m\ge2.$ Take a Puiseux parametrization of $(X,0)$ $$\varphi:(\mathbb{C},0)\longrightarrow(X,0)$$ given by $\varphi(u)=(u^m,\varphi_1(u))$ where $$\varphi_1(u)=a_1u^{\alpha_1}+a_2u^{\alpha_2}+\cdots,$$ with $m<\alpha_1<\alpha_2<\cdots,$ $a_i\ne0\in\mathbb{C}$ and $\alpha_i$ is a positive integer for all $i.$\\
    
    Define $b_0=m$ and $b_j=gcd(m,\alpha_1,\dots,\alpha_j)$ for $j\ge1.$ Note that the set of $b_j's$ forms a non-increasing sequence of positive integers that stabilizes at some $j_0$, with $b_j=1$ for all $j\ge j_0.$ Now, define $e_0=m,$ and for each $i\ge1,$ define $e_i=\alpha_{j_i}$ where $\alpha_{j_i}$ is such that $b_{j_i-1}>b_{j_i}.$ Since the sequence $(b_j)_j$ stabilizes, there is a finite number, say $k,$ of such $e_i's.$ The numbers $e_0,e_1,\dots,e_{k-1}$ are called the \textit{characteristic exponents} of the curve $(X,0).$

    \begin{remark}\label{milnor}
        (a) Throughout this work, we denote by $ce(X,0)$ the number of characteristic exponents of $(X,0).$

        \noindent(b) In \rm{\cite{Milnor}}\textit{, Milnor presented the following characterization for the Milnor number of an irreducible plane curve $X\subset\mathbb{C}^2$ with an isolated singular point at the origin. The Milnor number of $X$ at $0$ is defined by \begin{equation}\label{miln}
            \mu(X,0):=\dim_{\mathbb{C}}\dfrac{\mathbb{C}\left\{x,y\right\}}{\left<\frac{\partial h}{\partial x},\frac{\partial h}{\partial y}\right>}
        \end{equation}
         where $(X,0)=(V(h),0).$}
    \end{remark}
	
    \subsection{Double point space for corank 1 map germs}\label{df}
	
	$ \ \ \ \ $ The multiple point spaces for map germs from $(\mathbb{C}^n,0)$ to $(\mathbb{C}^p,0),$ with $n\le p,$ play an important role in the study of their geometry. In this section, we focus only on the double point space for corank 1 map germs with $p=n+1$, which will be fundamental for the later results in this work.
    
    We are interested in studying the double point space, denoted by $D^2(f),$ and its projection onto the source of $f$, denoted by $D(f).$ Roughly speaking, $D^2(f)$ is the set of points $(\textit{\textbf{x}},\textit{\textbf{x}}')\in\mathbb{C}^n\times\mathbb{C}^n$ such that $\textit{\textbf{x}}\ne\textit{\textbf{x}}'$ and $f(\textit{\textbf{x}})=f(\textit{\textbf{x}}'),$ or $\textit{\textbf{x}}$ is a singular point of $f.$ To view $D^2(f)$ as an analytic space, we need to provide it with an appropriate analytic structure. We follow the construction in \cite{[8]}, which applies to holomorphic maps from $\mathbb{C}^n$ to $\mathbb{C}^p,$ with $n\le p.$

    Another important space in the study of the topology of $f(\mathbb{C}^n)$ is the double point hypersurface $D(f)$  in the source. According to the literature \cite{[7]}, an appropriate analytic structure for this space is given by Fitting ideals. In what follows, $f_*\mathcal{O}_n$ denotes $\mathcal{O}_n$ viewed as an $\mathcal{O}_{n+1}-$module by composition with $f,$ and $Fitt_k(f_*\mathcal{O}_n)$ denotes the $k-$th Fitting ideal of $f_*\mathcal{O}_n.$ The following definition gives us precisely the analytic structure of these spaces.

    \begin{definition}
        Let $U\subset\mathbb{C}^n$ and $V\subset\mathbb{C}^{n+1}$ be open sets. Suppose the map $f:U\rightarrow V$ is finite, that is, holomorphic, closed, and finite-to-one. Let $\pi|_{D^2(f)}:D^2(f)\subset U\times U\rightarrow U$ be the restriction of the projection onto the first factor to $D^2(f)$. The double point space is the complex space $$D(f)=V(Fitt_0(\pi_*\mathcal{O}_{D^2(f)})).$$ As a set, we have $D(f)=\pi(D^2(f)).$
    \end{definition}

    \begin{remark}\label{resultant}
        If $f:(\mathbb{C}^n,0)\rightarrow(\mathbb{C}^{n+1},0)$ is a corank 1 finite map germ, then, up to a change of coordinates, we can write $f$ in the form $$f(\textbf{x},y)=(\textbf{x},p(\textbf{x},y),q(\textbf{x},y)),$$ where $\textbf{x}=(x_1,\dots,x_{n-1})\in\mathbb{C}^{n-1}, y\in\mathbb{C}, p$ and $q$ are elements of $\mathfrak{m}_n^2$ and $\mathfrak{m}_n$ denotes the maximal ideal of $\mathcal{O}_n.$ In this context, it follows that \begin{equation}\label{d2f}
            D^2(f)=V\left(\dfrac{p(\textbf{x},y)-p(\textbf{x},y')}{y-y'},\dfrac{q(\textbf{x},y)-q(\textbf{x},y')}{y-y'}\right)\subset\mathbb{C}^n\times\mathbb{C}.
        \end{equation} Moreover, denoting these divided differences by $\phi(\textbf{x},y,y')$ and $\psi(\textbf{x},y,y'),$ respectively, in \rm{\cite{MararJuan}} \textit{Marar e Nuño-Ballesteros observed that $D(f)$ is given by the resultant of $\phi$ and $\psi$ with respect to the variable $y',$ that is, $D(f)=V(Res_{\phi,\psi,y'}(\textbf{x},y)).$ For a definition and properties of the resultant of two polynomials, see} \rm{\cite{ideal}}.
    \end{remark}

    Another important definitions, where the terminology is due the first author \cite{[10]}, are the fold components of $D(f)$ and identification components of $D(f)$ definitions. 

    \begin{definition}
        Let $f:(\mathbb{C}^2,0)\rightarrow(\mathbb{C}^3,0)$ be a finitely determined map germ. Let $f:U\rightarrow V$ be a representative of $f$ and consider an irreducible component $D(f)^j$ of $D(f).$\\

        \noindent(a) If the restriction $f|_{D(f)^j}: D(f)^j\rightarrow V$ is generically 1-1, we say that $D(f)^j$ is an identification component of $D(f)$. In this case, there exists an irreducible component $D(f)^i$ of $D(f),$ with $i\ne j,$ such that $f(D(f)^i)=f(D(f)^j).$ We say that $D(f)^i$ is the associated identification component to $D(f)^j$ or that the pair $(D(f)^i,D(f)^j)$ is a pair of identification components of $D(f).$

        \noindent(b) If the restriction $f|_{D(f)^j}: D(f)^j\rightarrow V$ is generically 2-1, we say that $D(f)^j$ is a fold component of $D(f)$.
    \end{definition}

    \subsection{Finite determinacy and quasihomogeneous map germs}
	
	$ \ \ \ \ $ In this work, the concept of finite determinacy for map germs from $(\mathbb{C}^n,0)$ to $(\mathbb{C}^p,0)$ is of utmost importance. Thus, since we are interested in studying germs of mappings where $p=n+1$ we present the definition for this case.

    \begin{definition}
        (a) Two map germs $f,g:(\mathbb{C}^n,0)\rightarrow(\mathbb{C}^{n+1},0)$ are $\mathcal{A}$-equivalent,e denoted by $f\sim_{\mathcal{A}}g,$ if there exists germs of diffeomorphisms $\Phi:(\mathbb{C}^n,0)\rightarrow(\mathbb{C}^n,0)$ and $\Psi:(\mathbb{C}^{n+1},0)\rightarrow(\mathbb{C}^{n+1},0)$ such that $g=\Psi\circ f\circ\Phi.$

        \noindent (b) A map germ $f:(\mathbb{C}^n,0)\rightarrow(\mathbb{C}^{n+1},0)$ is $\mathcal{A}$-finitely determined (or simply finitely determined) if there exists a positive integer $k$  such that for every $g$ with $k$-jet satisfying $j^kg(0)=j^kf(0),$ we have $g\sim_{\mathcal{A}}f.$
    \end{definition}

    A characterization of the concept of finite determinacy is given by the Mather-Gaffney geometric criterion \cite[Th. 2.1]{Wall}. In the particular case where $n=2,$ this criterion states that, for a map germ $f:(\mathbb{C}^2,0)\rightarrow(\mathbb{C}^3,0),$ the finite determinacy of $f$ is equivalent to the existence of a finite representative $f:U\rightarrow V,$ with $U\subset\mathbb{C}^2, V\subset\mathbb{C}^3$ open neighborhoods of the origin of their respective spaces, such that $f^{-1}(0)=\left\{0\right\}$ and the restriction $f:U\setminus\left\{0\right\}\rightarrow V\setminus\left\{0\right\}$ is stable. This means that, outside the origin, the only singularities of $f$ are cross-caps (Whitney umbrellas), transversal double points, and triple points.

    It is also worth noting that, in the study of finite determinacy of a map germ $f:(\mathbb{C}^n,0)\rightarrow(\mathbb{C}^p,0),$ Marar and Mond \cite[Th. 2.14]{MararMond} presented necessary and sufficient conditions to ensure the finite determinacy of $f$ in terms of $D^2(f)$ and other multiple point spaces, provided that $f$ has corank 1. Moreover, in the case where $n=2$ and $p=3,$ Marar, Nuño-Ballesteros, and Peñafort-Sanchis \cite{[11]} extended this criterion to the corank 2 case, according to the following result:

    \begin{theorem}[\cite{[11]}, \cite{MararMond}]
        Let $f:(\mathbb{C}^2,0)\rightarrow(\mathbb{C}^3,0)$ be a finite and generically one-to-one map germ. Then $f$ is finitely determined if and only if the Milnor number of $D(f)$ at the origin is finite.
    \end{theorem}

Now, since we are interested in studying quasihomogeneous map germs in this work, it is appropriate to present a precise definition of this object.

\begin{definition}
    A polynomial $p(x_1,\dots,x_n)$ is said to be quasihomogeneous if there exists positive integers $w_1,\dots,w_n$ with no common factor and a positive integer $d$ such that $$p(k^{w_1}x_1,\dots,k^{w_n}x_n)=k^dp(x_1,\dots,x_n).$$ The integer $w_i$ is called the weight of the variable $x_i$ and $d$  is called the weighted degree of $p.$ In this case, we say that $p$ is of type $(d;w_1,\dots,w_n).$ We say that $f:(\mathbb{C}^n,0)\rightarrow(\mathbb{C}^p,0)$  is a quasihomogeneous map germ of type $(d_1,\dots,d_p;w_1,\dots,w_n)$ if each coordinate function $f_i$ is quasihomogeneous of type $(d_i;w_1,\dots,w_n).$ In particular, when $(n,p)=(2,3),$ we have that $f$ is quasihomogeneous of type $(d_1,d_2,d_3;w_1,w_2).$
\end{definition}

	\section{The notion of a $\mu_{\textbf{m},\textbf{k}}$-minimal plane curve germ}\label{sec3}
	
	$ \ \ \ \ \ $ The following lemma provides a formula for the Milnor number of an irreducible plane curve $(X,0)$ in terms of its characteristic exponents.

    \begin{lemma}\label{Lema1}
        Let $(X,0)$ be an irreducible plane curve germ of multiplicity $m\ge2,$ and let $e_0=m, e_1,\dots,e_k$ be the characteristic exponents of $(X,0)$. Denote $b_0=m$ and $b_i=gcd(e_0,e_1,\dots,e_i),$ for each $i\in\left\{1,\dots,k\right\}.$ Then \begin{equation}\label{formmu}
            \mu(X,0)=\sum_{i=1}^k(b_{i-1}-b_i)(e_i-1).
        \end{equation}
    \end{lemma}

    \begin{proof}
        The proof essentially follows by Theorem 6.12 and Proposition 7.5 of \cite{Abramo}.
    \end{proof}

\begin{example}\label{exx}
(a) Let us verify the effectiveness of the formula \eqref{formmu} in Lemma \rm{\ref{Lema1}}. \textit{Consider the curve $(X,0)=(V(h),0),$ where $h:(\mathbb{C}^2,0)\rightarrow(\mathbb{C},0)$ is defined by \begin{eqnarray*}
        h&=&y^8-2x^5y^4-8x^4y^5-26y^9+x^{10}-8x^9y+12x^8y^2+52x^5y^5+208x^4y^6\\
        &&-x^{11}-26x^{10}y+208x^9y^2-312x^8y^3+26x^{11}y.
    \end{eqnarray*} Note that $\varphi(u)=(u^8,u^{10}+u^{11})$ parametrizes the irreducible plane curve germ $(X,0).$ In the notation of Lemma} \rm{\ref{Lema1}}, \textit{$b_0=e_0=8, e_1=10$ and $e_2=11,$ while $b_1=2$ and $b_2=1.$ Then, using formula \eqref{formmu}, it follows that $\mu(X,0)=(8-2)(10-1)+(2-1)(11-1)=64.$ On the other hand, calculating the Milnor number of $(X,0)$ by hand using Remark \ref{milnor} (b), is not straightforward.}
    
    \noindent \textit{(b) Consider the irreducible plane curve germ $(X,0)$ given by the parametrization $$\varphi(u)=(u^8,u^{12}+u^{14}+u^{15}).$$ We have $e_0=b_0=8, e_1=12, e_2=14$ and $e_3=15,$ while $b_1=4, b_2=2$ and $b_3=1.$ By Lemma} \rm{\ref{Lema1}}, \textit{we conclude that $$\mu(X,0)=(8-4)(12-1)+(4-2)(14-1)+(2-1)(15-1)=84.$$ Now let $\psi(u)=(u^8,u^{20}+u^{22}+u^{21})$ be the parametrization of an irreducible plane curve germ $(Y,0).$ Again, by Lemma} \rm{\ref{Lema1}}, \textit{it follows that $\mu(Y,0)=140.$ Note that $\mu(X,0)<\mu(Y,0).$}
\end{example}

Now, we introduce the notion of a sequence of nested divisors.

    \begin{definition}
    (a) Let $m\ge2$ be an integer. We say that $$m=d_0>d_1>\cdots>d_{s-1}>d_s=1$$ is a sequence of nested divisors of length $s$ if $d_{i+1}$ divides $d_i$ for every $i=0,\dots,s-1.$

    \noindent (b) Define $\sigma:\mathbb{N}\rightarrow\mathbb{N}$ by $$\sigma(m):= \left\{\begin{array}{lc}
    1, & \text{if}\ \, m=1\\
    \text{the maximum length of a sequence of nested divisors of}\ m, & \text{if}\ \, m>1
    \end{array}\right.$$ Given an integer $m\ge2$, we say that a sequence of nested divisors is maximal if its length is $\sigma(m).$
\end{definition}

The following lemma not only proves that $\sigma$ is well-defined but also provides its characterization in terms of the prime factorization of $m.$ The proof is omitted, as this result follows directly from \cite[Th. 1.11]{Apostol}.

\begin{lemma}\label{comp}
    Let $m\ge2$ be an integer. Then:\\
    
        \noindent(a) Every sequence of nested divisors of $m$ can be extended to a maximal sequence.
        
        \noindent(b) Consider the prime factorization of $m,$ namely, $m=p_1^{\alpha_1}p_2^{\alpha_2}\cdots p_r^{\alpha_r}$ as its prime factorization, with $p_i<p_{i+1},$ for all $i=1,\dots,r-1.$ Then $\sigma(m)=\alpha_1+\alpha_2+\cdots+\alpha_r.$
    
\end{lemma}

\begin{proof}
    See, for instance, Theorem 1.11 in \cite{Apostol}.
\end{proof}

\begin{remark}
Clearly, $\sigma(m)=1,$ if and only if $m=1.$ Furthermore, maximal nested divisor sequences are not unique in general. In fact, take $m=30$ and note that $\sigma(30)=3.$ Finally, consider the sequences $$d_0=30, d_1=15, d_2=5, d_3=1$$ and $$d_0=30, d_1=6, d_2=2, d_3=1.$$ This leads us to consider which maximal divisor sequence is most appropriate for our purposes (see Definition \rm{\ref{min}}\textit{)}.
\end{remark}

The function $\sigma$ yields the maximal possible number of characteristic exponents for an irreducible plane curve with fixed multiplicity $m$. Let us consider the following lemma:

\begin{lemma}\label{max}
    Let $m\ge2$ be a fixed integer. Then a germ of an irreducible plane curve $(X,0)$ of multiplicity $m$ admits at most $\sigma(m)+1$ characteristic exponents.
\end{lemma}

\begin{proof}
    Indeed, suppose there exists a germ of an irreducible plane curve $(X,0)$ with $k$ characteristic exponents such that $k>\sigma(m)+1.$ Thus, there exists an integer $r\ge2$ such that $k=\sigma(m)+r.$ If $e_0=m, e_1,\dots,e_{\sigma(m)+r-2},e_{\sigma(m)+r-1}$ are the characteristic exponents of $(X,0)$, then
    \begin{eqnarray*}
    b_0&=&m\\
    b_1&=&gcd(b_0,e_1)=gcd(m,e_1)\\
    &\vdots&\\
    b_{\sigma(m)+r-2}&=&gcd(m,e_1,\dots,e_{\sigma(m)+r-3},e_{\sigma(m)+r-2})=gcd(b_{\sigma(m)+r-3},e_{\sigma(m)+r-2})\\
    b_{\sigma(m)+r-1}&=&gcd(m,e_1,\dots,e_{\sigma(m)+r-2},e_{\sigma(m)+r-1})=gcd(b_{\sigma(m)+r-2},e_{\sigma(m)+r-1})=1
    \end{eqnarray*} and we obtain the sequence of nested divisors of $m$ $$m=b_0>b_1>\cdots>b_{\sigma(m)+r-2}>b_{\sigma(m)+r-1}=1$$ with length $\sigma(m)+r-1>\sigma(m),$ contradicting the maximality of $\sigma(m).$ Therefore, $k\le\sigma(m)+1.$
\end{proof}\\

Motivated by the topology of the slices of the generic fibers in the counterexamples to Ruas’ conjecture, we introduce the notion of a \textit{$\mu_{\textbf{m},\textbf{k}}$-minimal plane curve germ}.

\begin{definition}\label{min}
Let $r,m_1,\dots,m_r,k_1,\dots,k_r,$ be fixed positive integers with $k_i\le\sigma(m_i)+1$ for every $i=1,\dots,r.$ Consider the set $\Lambda_{\textbf{m},\textbf{k}}$ of all plane curves with $r$ branches such that the $i$-th branch has multiplicity $m_i$ and $k_i$ characteristic exponents, for every $i=1,\dots,r,$ where $\textbf{m}=(m_1,\dots,m_r)$ and $\textbf{k}=(k_1,\dots,k_r).$ We say that $(X,0)\in\Lambda_{\textbf{m},\textbf{k}}$ is a $\mu_{\textbf{m},\textbf{k}}$-minimal plane curve if $$\mu(X,0)=\min_{(Y,0)\in\Lambda_{\textbf{m},\textbf{k}}}\mu(Y,0).$$ In particular, when $r=1,$ we write $\mu_{\textbf{m},\textbf{k}}$-minimal instead of $\mu_{m,k}$-minimal, where $m$ is the multiplicity and $k$ is the number of characteristic exponents of the curve.
\end{definition}

\begin{example}\label{ex8}
   (a) The cusp given by the parametrization $\varphi(u)=(u^2,u^3)$ is $\mu_{2,2}$-minimal.
   
   \noindent(b) In the Introduction, we presented the family $f_t(x,y)=(x,y^4,x^5y+xy^5+y^6+ty^7),$ which appears in \rm{\cite{Ruas}}. \textit{For each $t$, let $\gamma_t$ be the intersection of $f_t(\mathbb{C}^2)$ with a generic plane $H\subset\mathbb{C}^3.$ We have that $\gamma_t$ is $\mu_{4,3}$-minimal. However, $\gamma_0$ is not $\mu_{4,3}$-minimal.}
   
   \noindent\textit{(c) Example} \rm{\ref{exx}} \textit{(a) is an example of a $\mu_{8,3}$-minimal curve, while Example \rm{\ref{exx}} \textit{(b) is an example of a $\mu_{8,4}$-minimal curve} given by the parametrization $\varphi.$ Also in Example} \rm{\ref{exx}} \textit{(b), the plane curve given by $\psi$ is not $\mu_{8,4}$-minimal.}

   \noindent\textit{(d) Consider the plane curve germ $(X,0)=(X^1,0)\cup(X^2,0)$ where $(X^1,0)=(V(x^2-y^3),0)$ and $(X^2,0)=(V(y^2-x^3),0).$ One can verify that $(X,0)$ is $\mu_{(2,2),(2,2)}$-minimal.}
\end{example}

\subsection{Normal form of an irreducible $\mu_{\textbf{m},\textbf{k}}$-minimal plane curve germ}\label{normalform1}

$ \ \ \ \ \ $  From now on, we restrict our attention to the irreducible case, aiming to determine a suitable normal form for use with plane curve germs that are $\mu_{m,k}$-minimal. Clearly, if $m=1,$ then we are in the smooth case and a normal form for it is given by $\varphi(u)=(u,0).$ Therefore, it remains to study the case where $m\ge2.$ Let us now consider the result for $k=2.$

    \begin{proposition}\label{Teo1}
    Let $m\ge2$ be a fixed integer. Then, a normal form of a $\mu_{m,2}$-minimal irreducible plane curve germ is given by $\varphi(u)=(u^m,u^{m+1}).$
\end{proposition}

\begin{proof}
    Let $(X,0)$ be the plane curve germ parametrized by $\varphi(u)=(u^m,u^{m+1}),$ and consider a germ of a plane curve $(Y,0)$ with multiplicity $m$ and two characteristic exponents, $m$ and $n,$ with $n\ge m+1.$ It follows from Lemma \ref{Lema1} that $\mu(Y,0)=(b_0-b_1)(e_1-1)$, where $b_0=m, b_1=gcd(m,n)=1$ and $e_1=n.$ Hence, \begin{equation}\label{1}
        \mu(Y,0)=(m-1)(n-1)\ge(m-1)m=\mu(X,0),
    \end{equation} where the last equality in \eqref{1} again follows by Lemma \ref{Lema1}, which concludes the proof.
\end{proof}

\begin{corollary}\label{primo}
    Let $(X,0)$ be a $\mu_{m,k}$-minimal irreducible plane curve germ. If $m$ is a prime integer, then $k=2$ and a normal form of $(X,0)$ is given by $\varphi(u)=(u^{m},u^{m+1}).$
\end{corollary}

\begin{proof}
    Since $m$ is prime, it follows that $gcd(m,e_1)=1,$ where $e_1$ is the second characteristic exponent of $(X,0).$ Therefore, $k=2$ and we obtain a normal form of $(X,0)$ by Proposition \ref{Teo1}.
\end{proof}\\

In particular, Corollary \ref{primo} ensures that, if $m=3,$ then a normal form of a $\mu_{3,2}$-minimal irreducible plane curve germ is given by $\varphi(u)=(u^3,u^4).$ Before presenting our first main result, let us consider the case of three characteristic exponents, which will provide insight into how to proceed in the general case.

\begin{proposition}\label{k3}
    Let $m\ge4$ be a composite integer and let $p$ be the smallest prime divisor of $m$. Then a normal form of the irreducible plane curve germ $(X,0)$ that is $\mu_{m,3}$-minimal is given by $$\varphi(u)=(u^m,u^{m+p}+u^{m+p+1}).$$
\end{proposition}

\begin{proof}
    We aim to show that the normal form $$\varphi(u)=(u^m,u^{m+p}+u^{m+p+1})$$ of $(X,0)$ is $\mu_{m,3}$-minimal. Let $(Y,0)$ be an irreducible plane curve germ of multiplicity $m$ with three characteristic exponents,  $e_0=m, e_1$ and $e_2.$ We want to show that $\mu(X,0)\le\mu(Y,0).$ To do this, consider an irreducible plane curve germ $(Y_0,0)$ with multiplicity $m$ and characteristic exponents $e_0=m, e_1$ and $e_1+1.$ Since $e_2>e_1$, by Lemma \ref{Lema1}, we have that $\mu(Y_0,0)\le \mu(Y,0).$ Thus, it is enough to show that $\mu(X,0)\le\mu(Y_0,0).$ Let us now examine the relationship between the characteristic exponents of $(Y_0,0)$ and $(X,0).$ First, we must have that $m+p\le e_1.$ Indeed, since $e_1>m,$ there exists $\alpha_1>0$ such that $e_1=m+\alpha_1.$ Now, denoting by $b_1=gcd(m,e_1),$ it follows that $b_1$ divides $\alpha_1.$ Moreover, $b_1>1.$ Then, if $q$ is the smallest prime divisor of $b_1,$ we must have $p\le q,$ since $q$ divides $m.$ Thus, we have $$m+p\le m+\alpha_1=e_1.$$ Now, let us examine the two cases: if $e_1=m+p,$ then $(X,0)$ and $(Y_0,0)$ have the same characteristic exponents and follows the result. If $e_1>m+p,$ then $e_1\ge m+p+1$ and by Lemma \ref{Lema1} we have that $\mu(X,0)\le\mu(Y_0,0).$ With this, we conclude the proof.
\end{proof}\\

Let us now present our first main result. We consider the case where $m\ge4$ is composite integer, which gives a normal form for $\mu_{m,k}-$minimal curves in the irreducible case.

\begin{theorem}\label{normal}
    Let $m\ge4$ be a composite integer with prime factorization $m=p_1p_2\cdots p_r$ with $p_1\le p_2\le\cdots\le p_r.$ If $k\in\left\{3,\dots,r+1\right\},$ then a normal form of the irreducible plane curve germ $(X,0)$ that is $\mu_{m,k}$-minimal is given by $$\varphi(u)=(u^m,u^{m+d_1}+u^{m+d_1+d_2}+\cdots+u^{m+d_1+d_2+\cdots+d_{k-2}}+u^{m+d_1+d_2+\cdots+d_{k-2}+1}),$$ where $$d_0=m>d_1>d_2>\cdots>d_{k-2}>d_{k-1}=1$$ is the nested divisor sequence of $m$ such that $d_i=p_1p_2\cdots p_{(k-1)-i}$ for $i=1,\dots,k-2.$
\end{theorem}

\begin{proof}
    The case $k=3$ was established in Proposition \ref{k3}. Let’s do the case $k=4$ and the general case follows analogously with the appropriate adjustments. We aim to show that the normal form $$\varphi(u)=(u^m,u^{m+p_1p_2}+u^{m+p_1p_2+p_1}+u^{m+p_1p_2+p_1+1})$$ of $(X,0)$ is $\mu_{m,4}$-minimal. Let $(Y,0)$ be an irreducible plane curve germ of multiplicity $m$ with four characteristic exponents, say $e_0=m, e_1, e_2$ and $e_3.$ We want to show that $\mu(X,0)\le\mu(Y,0).$ To do this, consider an irreducible plane curve germ $(Y_0,0)$ with multiplicity $m$ and characteristic exponents $e_0=m, e_1, e_2$ and $e_2+1.$ Since $e_3>e_2,$ we know $e_3\ge e_2+1.$ Now, writing $b_0=m, b_1=gcd(m,e_1), b_2=gcd(m,e_1,e_2)$ which are the same $gcd's$ for the characteristic exponents of $(Y_0,0).$ Then, by Lemma \ref{Lema1}, we have:
    \begin{eqnarray*}
        \mu(Y,0)&=&(b_0-b_1)(e_1-1)+(b_1-b_2)(e_2-1)+(b_2-b_3)(e_3-1)\\
        &\ge&(b_0-b_1)(e_1-1)+(b_1-b_2)(e_2-1)+(b_2-b_3)((e_2+1)-1)\\
        &=&\mu(Y_0,0).
    \end{eqnarray*} Thus, it is enough to show that $\mu(X,0)\le\mu(Y_0,0).$ Let us now examine the relation between the characteristic exponents of $(Y_0,0)$ and $(X,0).$ First, we must have that $m+p_1p_2\le e_1.$ Indeed, since $e_1>m,$ there exists $\alpha_1>0$ such that $e_1=m+\alpha_1.$ Now, since $b_1=gcd(m,e_1),$ it follows that $b_1$ divides $\alpha_1.$ Moreover, we have $b_1>b_2>1,$ so $b_1$ must have at least two distinct prime divisors in its prime factorization. There exist primes $q_1, q_2$ with $q_1 \le q_2$ such that $q_i \mid b_1$ for $i=1,2$. Thus, these primes $q_1$ and $q_2$ must divide $m$ and it follows that $p_1\le q_1$ and $p_2\le q_2.$ Therefore, $$m+p_1p_2\le m+q_1q_2\le m+b_1\le m+\alpha_1=e_1.$$ Now, let us examine the two cases:\\
   
        \noindent\underline{Case 1}: If $e_1=m+p_1p_2,$ then we must have $e_2\ge m+p_1p_2+p_1.$ To verify this, it is sufficient to apply an argument analogous to the one used to ensure that $e_1\ge m+p_1p_2.$ Thus, if $e_2=m+p_1p_2+p_1,$ then $(X,0)$ and $(Y_0,0)$ have the same characteristic exponents and therefore the same Milnor number. If $e_2>m+p_1p_2+p_1,$ consider the normal form $$\phi(u)=(u^m,u^{m+p_1p_2}+u^{e_2}+u^{e_2+1})$$ of $(Y_0,0)$ and deform it by a 1-parameter family $\phi_t$ defined by $$\phi_t(u)=(u^m,u^{m+p_1p_2}+tu^{m+p_1p_2+p_1}+tu^{m+p_1p_2+p_1+1}+u^{e_2}+u^{e_2+1}).$$ Then it follows that $\mu(X,0)=\mu((Y_0)_t,0)\le\mu(Y_0,0),$ for sufficiently small $t\ne0.$\\

        \noindent\underline{Case 2}: If $e_1>m+p_1p_2,$ then we must have $e_1\ge m+p_1p_2+p_1,$ using the same reasoning as before. If $e_1=m+p_1p_2+p_1,$ we define the 1-parameter deformation $$\phi_t(u)=(u^m,tu^{m+p_1p_2}+u^{m+p_1p_2+p_1}+tu^{m+p_1p_2+p_1+1}+u^{e_2}+u^{e_2+1})$$ of the normal form $\phi$ of $(Y_0,0)$ and we obtain $\mu(X,0)=\mu((Y_0)_t,0)\le\mu(Y_0,0),$ for sufficiently small $t\ne0.$ If $e_1>m+p_1p_2+p_1,$ we use the 1-parameter deformation $$\phi_t(u)=(u^m,tu^{m+p_1p_2}+tu^{m+p_1p_2+p_1}+tu^{m+p_1p_2+p_1+1}+u^{e_1}+u^{e_2}+u^{e_2+1})$$ of the normal form $\phi$ of $(Y_0,0)$ and again conclude that $\mu(X,0)=\mu((Y_0)_t,0)\le\mu(Y_0,0).$\\
     
     With this, we conclude the case $k=4.$
\end{proof}

\begin{remark}
    (a) Note that, in addition to the normal form, we also present the characteristic exponents of the irreducible plane curve germs that are $\mu_{m,k}$-minimal. Thus, if $(X,0)\in\Lambda_{m,k}$ has a characteristic exponent different from those presented in Theorem \rm{\ref{normal}}\textit{, then $(X,0)$ is not $\mu_{m,k}$-minimal.}

    \noindent\textit{(b) In the proof of Theorem} \rm{\ref{normal}}\textit{, an interesting fact is that if $(X,0)$ is a plane curve germ of multiplicity $m$ and $k$ characteristic exponents (not necessarily $\mu_{m,k}$-minimal), it is always possible to deform $(X,0)$ into a $\mu_{m,k}$-minimal one. For example, consider the plane curve $(X,0)$ parametrized by $\varphi(u)=(u^4,u^6+u^9).$ This curve is not $\mu_{4,3}$-minimal. However, consider the 1-parameter unfolding of $\varphi$ given by $$\varphi_t(u)=(u^4,u^6+u^9+tu^{7}).$$ For $t\ne0$ sufficiently small, the plane curve germ $(X_t,0)$ parametrized by $\varphi_t$ is $\mu_{4,3}$-minimal.}
\end{remark}

\begin{corollary}\label{defmin}
    Let $(X,0)$ be an irreducible plane curve germ of multiplicity $m$ and $k$ characteristic exponents, say $e_0=m,e_1,\dots,e_{k-1}.$ If $\varphi$ denotes a normal form of $(X,0),$ then there exists a deformation $\varphi_t$ of $\varphi$ such that $\varphi_t$ is a normal form of a $\mu_{m,k}$-minimal irreducible plane curve germ, for all $t\ne0.$
\end{corollary}

\begin{proof}
    The proof is carried out throughout the proof of Theorem \ref{normal}.
\end{proof}\\

Below, we present Table \ref{table} with some examples of normal forms and their respective Milnor numbers for $\mu_{36,k}$-minimal plane curves $2\le k\le5$.

\begin{table}[H]
\centering
{\def\arraystretch{2.0}\tabcolsep=22pt 

\begin{tabular}{ c || c ||  c }

\hline 
\rowcolor{lightgray}
 \textbf{Plane curve} & \textbf{Normal form} & \textbf{Milnor number}  \\
			  
\hline

$\mu_{36,2}$-minimal  & $\varphi(u)=(u^{36},u^{37})$ & $1260$ \\ \hline
     
$\mu_{36,3}$-minimal    & $\varphi(u)=(u^{36},u^{38}+u^{39})$  & $1296$ \\
\hline

$\mu_{36,4}$-minimal & $\varphi(u)=(u^{36},u^{40}+u^{42}+u^{43})$ & $1372$   \\

\hline

$\mu_{36,5}$-minimal & $\varphi(u)=(u^{36},u^{48}+u^{52}+u^{54}+u^{55})$ & $1696$   \\

\hline
\end{tabular}
}
\caption{Normal forms and Milnor numbers for $\mu_{36,k}$-minimal plane curves}\label{table}
\end{table}

\subsection{Upper semicontinuity for $\mu_{\textbf{m},\textbf{k}}$-minimal plane curve germs}
 
$ \ \ \ \ \ $ Unlike invariants like the Milnor number or the multiplicity, the number of characteristic exponents of an irreducible plane curve does not have the upper semicontinuity property. That is, the number of characteristic exponents can increase under a deformation of the curve. For example, we can consider the curve given by the parametrization $\varphi(u)=(u^4,u^9)$ and the deformation of $\varphi$ given by $$\varphi_t(u)=(u^4,tu^6+tu^7+u^9).$$ Note that the plane curve parametrized by $\varphi$ has two characteristic exponents, while the plane curve parametrized by $\varphi_t$ has three for any $t\ne0.$ However, in the case of $\mu_{m,k}$-minimal irreducible plane curves, the number of characteristic exponents does acquire this upper semicontinuity property as a consequence of the normal form established in Theorem \ref{normal}, provided the family is equimultiple.

\begin{corollary}\label{upper}
    Let $(X,0)$ be an irreducible plane curve germ, with multiplicity $m\ge2,$ and $k$ characteristic exponents. Suppose that $(X,0)$ is $\mu_{m,k}$-minimal. Then:\\

    \noindent (a) If $m$ is composite and $k\ge4,$ then $\mu(X,0)=m(m-2+d_1)+\sum_{j=1}^{k-3}d_j(d_{j+1}-1),$ where $d_j$ is given as in Theorem \rm{\ref{normal}}\textit{, for all $j=1,\dots, k-3$.}

    \noindent \textit{(b) If $(Y,0)$ is another irreducible plane curve germ with multiplicity $m$ and $k'$ characteristic exponents, with $k<k',$ then $\mu(X,0)<\mu(Y,0)$.}

    \noindent \textit{(c) Let $\varphi:(\mathbb{C},0)\rightarrow(\mathbb{C}^2,0)$ be a Puiseux parametrizations of $(X,0)$ and consider $F=(\varphi_t,t)$ a 1-parameter unfolding of $\varphi.$ Denote by $(X_t,0)$ the image of $\varphi_t.$ Suppose that $(X_t,0)$ is irreducible for any sufficiently small $t$. If $m(X_t,0)$ is constant along the parameter space, then $ce(X_t,0)\le ce(X,0),$ for all sufficiently small $t.$}
\end{corollary}

\begin{proof}
    (a) Using the normal form established in Theorem \ref{normal} and the formula \eqref{formmu} from Lemma \ref{Lema1}, we obtain that
    \begin{eqnarray*}
        \mu(X,0)&=&\sum_{j=1}^{k-1}(d_{j-1}-d_j)(d_0+\dots+d_j-1)\\
        &=&\sum_{j=1}^{k-1}(d_{j-1}^2-d_j^2)+\sum_{j=0}^{k-3}d_j(d_{j+1}-1)-(d_0-1)\\
        &=&m(m-2+d_1)+\sum_{j=1}^{k-3}d_j(d_{j+1}-1).
    \end{eqnarray*}

    \noindent (b) It is sufficient to consider the case where $k'=k+1.$ If $k=2,$ by Lemma \ref{Lema1} we already know that $\mu(X,0)=m(m-1)$, and it can be verified that $\mu(Y,0)=m(m-2+d_1).$ Thus, note that $\mu(X,0)<\mu(Y,0).$ For $k=3,$ we have $\mu(X,0)=m(m-2+d_1)$ while $\mu(Y,0)=m(m-2+d_1)+d_1(d_2-1).$ Clearly, $\mu(X,0)<\mu(Y,0).$ For $k\ge4,$ it is sufficient to use item (a) and verify that $\mu(X,0)<\mu(Y,0)$.

    \noindent (c) We have that $(X_t,0)$ is an irreducible plane curve germ, for all $t$. Suppose that, in a sufficiently small neighborhood of the origin, $ce(X,0)<ce(X_t,0)$ for $ t\ne0.$ Using item (b), we obtain $\mu(X,0)<\mu(X_t,0),$ which contradicts the fact that the Milnor number has the upper semicontinuity property for sufficiently small $t$. Therefore, we must have $ce(X_t,0)\le ce(X,0).$
\end{proof}

\subsection{Normal form of a $\mu_{\textbf{m},\textbf{k}}$-minimal plane curve germ with several branches}\label{normalform}

$ \ \ \ \ \ $ Now we explore the reducible case of $\mu_{\textbf{m},\textbf{k}}$-minimal plane curves in order to present a normal form for this class of curves. To do this, we use the normal form obtained in the irreducible case. The following proposition will provide a characterization of $\mu_{\textbf{m},\textbf{k}}$-minimal plane curve germs to help us present their normal form. In what follows, we say that $(X^i,0)$ and $(X^j,0)$ are transversal if they does not have a common tangent. To obtain the next result, let us recall the notion of intersection multiplicity of $f$ and $g$ in $\mathbb{C}\left\{x,y\right\}$ at the origin, as in \cite{Greuel}. Essentially, given $f,g\in\mathbb{C}\left\{x,y\right\},$ the intersection multiplicity of $f$ and $g$ at the origin, denoted by $i(f,g),$ is defined by  $$i(f,g)=\dim_{\mathbb{C}}\dfrac{\mathbb{C}\left\{x,y\right\}}{\langle f,g\rangle}.$$

\begin{proposition}\label{Prop}
    Consider fixed positive integers $r,m_1,\dots,m_r,k_1,\dots,k_r$ with $r\ge2,$ set $\textbf{m}=(m_1,\dots,m_r)$ and $\textbf{k}=(k_1,\dots,k_r).$ Consider a plane curve germ $(X,0)=(X^1,0)\cup\dots\cup(X^r,0)$ with $(X^i,0)$ of multiplicity $m_i$ and $k_i$ characteristic exponents. Then, $(X,0)$ is $\mu_{\textbf{m},\textbf{k}}$-minimal, if and only if each $(X^i,0)$ is $\mu_{m_i,k_i}$-minimal.
\end{proposition}

\begin{proof}
    Let us consider the case $r=2$ and the general case follows analogously. Consider the plane curve germ $(X,0)=(X^1,0)\cup(X^2,0)$, where each component $(X^i,0)$ is an irreducible germ with multiplicity $m_i$ and with $k_i$ characteristic exponents. By \cite[Corollary 1.2.3]{Buchweitz}, it follows that \begin{equation}\label{bg}
        \mu(X,0)=\mu(X^1,0)+\mu(X^2,0)+2i(X^1,X^2)-1.
    \end{equation} Suppose that $(X,0)$ is $\mu_{\textbf{m},\textbf{k}}$-minimal. If $(X^1,0)$ is not $\mu_{m_1,k_1}$-minimal, then by Corollary \ref{defmin}, given a parametrization $\varphi_1$ of $(X^1,0),$ there exists a deformation $(\varphi_1)_t$ of $\varphi_1$ such that $(X^1_t,0)$ is $\mu_{m_1,k_1}$-minimal for all $t\ne0$, where $(X^1_t,0)$ is the plane curve germ parametrized by $(\varphi_1)_t.$ Moreover, we know that $\mu(X^1_t,0)<\mu(X^1,0),$ for all $t\ne0.$ Hence, since this deformation of $(X^1,0)$ induces a deformation of $(X,0),$ it follows from \eqref{bg} that $\mu(X_t,0)<\mu(X,0)$ along this deformation for all $t\ne0,$ which contradicts the minimality of $\mu(X,0).$ Therefore, $(X^1,0)$ is $\mu_{m_1,k_1}$-minimal. By an analogous argument, one concludes that $(X^2,0)$ is $\mu_{m_2,k_2}$-minimal. Now, to show that $(X^1,0)$ and $(X^2,0)$ are transversal, consider $\varphi_1$ and $\varphi_2$ as parametrizations of $(X^1,0)$ and $(X^2,0),$ respectively. By performing a 1-parameter deformation of the curves such that, for sufficiently small $t\ne0$, $(X^1_t,0)$ and $(X^2_t,0)$ are transversal, it follows from \eqref{bg} that $$\mu(X,0)\ge\mu(X^1_t,0)+\mu(X^2_t,0)+2i(X^1_t,X^2_t)-1.$$ It follows from $i(X^1,X^2)=i(X^1_t,X^2_t)=m_1m_2,$ for sufficiently small $t\ne0$. Therefore, $(X^1,0)$ and $(X^2,0)$ are transversal. Conversely, suppose that $(X^i,0)$ is $\mu_{m_i,k_i}$-minimal for $i=1,2$ and that $(X^1,0)$ and $(X^2,0)$ are transversal. Given $(Y,0)=(Y^1,0)\cup(Y^2,0)\in\Lambda_{\textbf{m},\textbf{k}},$ by the minimality of the Milnor number of $(X^i,0),$ it follows that $\mu(X^i,0)\le\mu(Y^i,0),$ for $i=1,2.$ Moreover, since $i(Y^1,Y^2)\ge m_1m_2=i(X^1,X^2),$ we conclude from \eqref{bg} that $\mu(X,0)\le\mu(Y,0).$ Therefore, $(X,0)$ is $\mu_{\textbf{m},\textbf{k}}$-minimal. For the general case, it is sufficient to observe that in \cite[Corollary 1.2.3]{Buchweitz} we have $$\mu(X,0)=1+\left(\sum_{1\le i\le r}(\mu(X^i,0)-1)\right)+2\sum_{1\le i<j\le r}i(X^i,X^j).$$ Thus, with analogous arguments, we conclude the result.
\end{proof}\\

Using Proposition \ref{Prop}, we obtain a normal form for $\mu_{\textbf{m},\textbf{k}}$-minimal plane curve germs, valid for $r\ge2.$

\begin{theorem}\label{Normal}
    Let $r\ge2$ be an integer and $(X,0)=(X^1,0)\cup\dots\cup(X^r,0)$ be a germ of plane curve such that each $(X^i,0)$ is of multiplicity $m_i$ and $k_i$ characteristic exponents, for all $i=1,\dots,r.$ Consider the multi-indices $\textbf{m}=(m_1,\dots,m_r)$ and $\textbf{k}=(k_1,\dots,k_r).$ If $(X,0)$ is $\mu_{\textbf{m},\textbf{k}}$-minimal, then $(X,0)$ admits an explicit normal form, given by the set of normal forms of all $(X^i,0).$
\end{theorem}

\begin{proof}
 Let $(X,0)=(X^1,0)\cup\dots(X^r,0)$ be a $\mu_{\textbf{m},\textbf{k}}$-minimal plane curve germ. By Proposition \ref{Prop}, each $(X^i,0)$ is a $\mu_{m_i,k_i}$-minimal irreducible plane curve germ and the branches are pairwise transverse. Thus, for each $i=1,\dots,r,$ we have that $(X^i,0)$ admits an explicit normal form given by $\phi_i(u)=(u,0)$ for $m_i=1$ or, by Theorem \ref{normal}, $(X^i,0)$ admits an explicit normal form given by $$\phi_i(u)=(u^{m_i},u^{m_i+d_{1,i}}+u^{m_i+d_{1,i}+d_{2,i}}+\cdots+u^{m_i+d_{1,i}+d_{2,i}+\cdots+d_{k-2,i}}+u^{m_i+d_{1,i}+d_{2,i}+\cdots+d_{k-2,i}+1}),$$ where $$d_{0,i}=m>d_{1,i}>\dots>d_{k-2,i}>d_{k-1,i}=1$$ is a sequence of nested divisors of $m_i$ as in Theorem \ref{normal}, if $m_i\ge2.$ But from Proposition \ref{Prop}, we know that the branches of $(X,0)$ are pairwise transversal. Therefore, we can consider the normal forms of $(X^i,0)$ given by $\varphi_i(u)=(u,\alpha_iu),$ if $m_i=1,$ and $$\varphi_i(u)=(u^{m_i},\alpha_iu^{m_i}+u^{m_i+d_{1,i}}+u^{m_i+d_{1,i}+d_{2,i}}+\cdots+u^{m_i+d_{1,i}+d_{2,i}+\cdots+d_{k-2,i}}+u^{m_i+d_{1,i}+d_{2,i}+\cdots+d_{k-2,i}+1}),$$ if $m_i\ge2,$ with $\alpha_i\in\mathbb{C}$ for all $i=1,\dots,r,$ and with the $\alpha_i's$ pairwise distinct. In this way, we ensure the transversality of the branches of $(X,0)$ and the set $\left\{\varphi_i\right\}_{i=1}^r$ gives us a normal form of $(X,0),$ as desired.
\end{proof}

\begin{example}
    (a) In item (d) of Example \rm{\ref{ex8}}\textit{, we present the plane curve germ $(X,0)=(V(x^2-y^3),0)\cup(V(y^2-x^3),0),$ which is $\mu_{(2,2),(2,2)}$-minimal. For this curve germ, we can consider the normal form $\left\{\varphi_1,\varphi_2\right\}$ where $\varphi_1(u)=(u^2,u^3)$ is a normal form of the first branch and $\varphi_2(u)=(u^2,u^2+u^3)$ is a normal form of the second branch.}

    \noindent\textit{(b) Consider a plane curve germ $(X,0)$ $\mu_{(2,4),(2,3)}$-minimal. Thus, $(X^1,0)$ is $\mu_{2,2}$-minimal and $(X^2,0)$ is $\mu_{4,3}$-minimal. Moreover, $(X^1,0)$ and $(X^2,0)$ are transversal. Therefore, a normal form of $(X,0)$ is given by $\left\{\varphi_1,\varphi_2\right\}$ where $\varphi_1(u)=(u^2,u^3)$ and $\varphi_2(u)=(u^4,u^4+u^6+u^7).$}
\end{example}
	
\section{New counterexamples to Ruas’ conjecture}\label{newex}	
	
	$ \ \ \ \ $ In this section, we work with $f:(\mathbb{C}^2,0)\rightarrow(\mathbb{C}^3,0)$ a finitely determined, quasihomogeneous, corank 1 map germ. We also present an infinite list of counterexamples to the implication $(b) \Rightarrow (c)$ in Ruas’ conjecture. Furthermore, we present a counterexample showing that the equivalence between $(c)$ and $(d)$ in Ruas's conjecture does not hold, in general. However, as an application of the results obtained in the previous section, we begin this section by presenting some examples of map germs whose transverse slice of $f$ is $\mu_{m,k}$-minimal. To do this, we use results from \cite{slice}. In the sequel, we recall the notion of transverse slice, originally introduced in \cite{MararJuan}.

\begin{definition}\label{slice}
Let $f:(\mathbb{C}^2,0)\rightarrow(\mathbb{C}^3,0)$ be a finite map germ.\\

    \noindent (a) We say that a plane $H\subset\mathbb{C}^3$ through the origin is generic for $f$ if the following conditions are satisfied:\\
    
        (a.1) $H\cap df_0(\mathbb{C}^2)=\left\{(0,0,0)\right\},$ where $df_0(\mathbb{C}^2)$ denotes the image of the differential of $f$ at the origin.
        
        (a.2) $H\cap f(D(f))=\left\{(0,0,0)\right\},$ where $f(D(f))$ is the image of the double point curve $D(f)$ by $f.$
        
        (a.3) $H\cap C_0(f(D(f)))=\left\{(0,0,0)\right\},$ where $C_0(f(D(f)))$ denotes the Zariski tangent cone of $f(D(f)).$\\

        \noindent(b) We define a transverse slice of $f$, and denotes it by $\gamma,$ as the intersection of the image of $f$ with a generic plane $H\subset\mathbb{C}^3$ for $f.$
\end{definition}

\begin{remark}
    In general, the analytic type of the curve $\gamma$ may depend on the choice of the coordinates and the generic plane. However, the topological type does not depend on this choice (see \rm{\cite{Rudd}}\textit{) and it is usual to call it the transverse slice of $f.$}
\end{remark}

\begin{example}\label{exslice}
    Consider the finitely determined map germ $f:(\mathbb{C}^2,0)\rightarrow(\mathbb{C}^3,0)$ given by $f(x,y)=(x,y^2,xy).$ One verifies that the plane $H=V(X-Y)\subset\mathbb{C}^3$ is generic for $f.$
    \begin{figure}[H]
        \centering
        \includegraphics[scale=0.2]{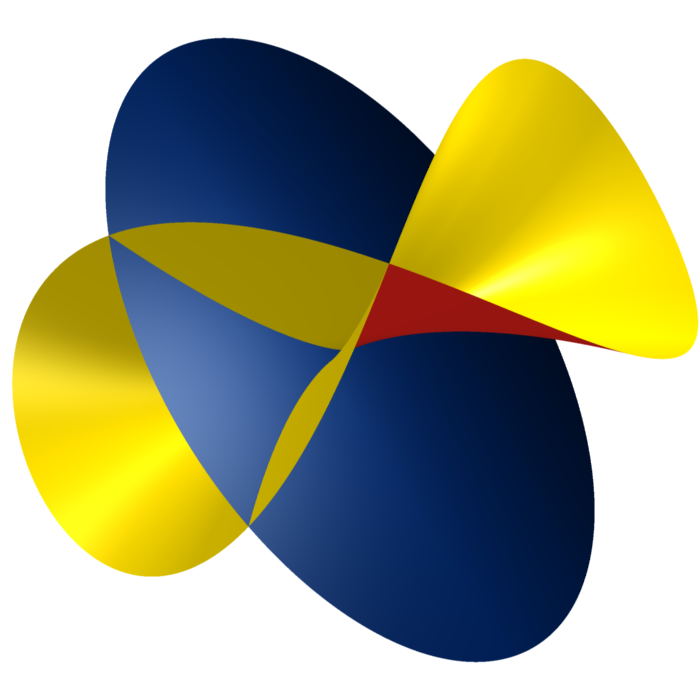}
        \caption{Transverse slice for the cross-cap (real points).}
        \label{Figure1}
    \end{figure} Denoting by $\widetilde{\gamma}$ the preimage of the slice $\gamma,$ that is, $\widetilde{\gamma}=f^{-1}(f(\mathbb{C}^2)\cap H)=f^{-1}(H),$ we obtain that $\widetilde{\gamma}=V(x-y^2).$ A parametrization for $\widetilde{\gamma}$ is given by $\widetilde{\phi}(u)=(u^2,u).$ Thus, we can parametrize the slice $\gamma$ of $f$ by composing $f$ with $\widetilde{\phi},$ that is, $$\phi(u)=f\circ\widetilde{\phi}(u)=(u^2,u^2,u^3)$$ is a parametrization of the slice $\gamma$ of $f.$ Performing a coordinate change in the target, it follows that $\phi(u)=(0,u^2,u^3)$ and $\gamma$ is topologically equivalent to a cusp.
\end{example}

\begin{remark}
   When working with unfoldings $F=(f_t,t)$, it is always possible to obtain a generic plane $H$ for $f$ such that $H$ is also generic for $f_t$ for all $t$ sufficiently small.
\end{remark}

In \cite{slice}, the first author characterized the topology of the transverse slice of a finitely determined, quasihomogeneous, corank 1 map germ from $(\mathbb{C}^2,0)$ to $(\mathbb{C}^3,0).$ The first author proved that the transverse slice has only two  or three exponents characteristics. By results in \cite{slice}, we consider the additional assumption that the transverse slice is $\mu_{m,k}$-minimal to obtain some information in this setting. We present examples with $\mu_{m,k}$-minimal transverse slice for each case.

    \begin{lemma}\label{slice1}
        Let $f:(\mathbb{C}^2,0)\rightarrow(\mathbb{C}^3,0)$ be a finitely determined, quasihomogeneous, corank 1 map germ. Write $f$ in quasihomogeneous form $$f(x,y)=(x,y^m+xp(x,y),y^n+xq(x,y))$$ of type $(w_1,d_2,d_3;w_1,w_2)$ with $d_2\le d_3,$ as in \rm{\cite[\textit{Lemma} 2.11]{normalform}}\textit{. Suppose that $w_1\le d_2$ and $gcd(m,n)=1.$ If the transverse slice $\gamma$ of $f$ is $\mu_{m,2}$-minimal, then $n=m+1$ and the characteristic exponents are $m$ and $m+1.$}
    \end{lemma}

    \begin{proof}
        By \cite[Prop. 3.10]{slice}, the characteristic exponents of the transverse slice $\gamma$ of $f$ are $m$ and $n.$ But $\gamma$ is $\mu_{m,2}$-minimal. Then the characteristic exponents of $\gamma$ are $m$ and $m+1,$ that is, $n=m+1.$
    \end{proof}

    \begin{example}
        We can find examples of transverse slice $\mu_{m,2}$-minimal as in Lemma \rm{\ref{slice1}} \textit{in Mond's list} \rm{\cite[\textit{p.} 378]{[8]}}. \textit{One can verify that cross-cap and $S_k$ (for $k\ge1$) singularities in Mond's list have $\mu_{2,2}$-minimal transverse slice. Furthermore, $T_4$ and $P_3$ singularities in Mond's list have $\mu_{3,2}$-minimal transverse slice. The other singularities in Mond’s list do not have $\mu_{m,2}$-minimal transverse slice.}
    \end{example}

    \begin{remark}
        Note that in Lemma \rm{\ref{slice1}} \textit{we can construct examples of quasihomogeneous map germs with $\mu_{m,2}$-minimal transversal slice, in which the weights may be equal or distinct. For example, consider the cross-cap map $f(x,y)=(x,y^2,xy).$ We have that $f$ is a finitely determined, homogeneous, corank 1 map germ with $\mu_{2,2}$-minimal transverse slice. On the other hand, consider the map germ $g(x,y)=(x,y^2,y^3+x^4y),$ the $S_3$ singularity in Mond's list} \rm{\cite[\textit{p.} 378]{[8]}}. \textit{This is a finitely determined, quasihomogeneous, corank 1 map germ with $w_1=1$ and $w_2=2,$ whose transverse slice is $\mu_{2,2}$-minimal.}
    \end{remark}

    \begin{lemma}\label{slice2}
        Let $f:(\mathbb{C}^2,0)\rightarrow(\mathbb{C}^3,0)$ be a finitely determined, quasihomogeneous, corank 1 map germ. Write $f$ in quasihomogeneous form $$f(x,y)=(x,y^m+xp(x,y),y^n+xq(x,y))$$ of type $(w_1,d_2,d_3;w_1,w_2)$ with $d_2\le d_3,$ as in \rm{\cite[\textit{Lemma} 2.11]{normalform}}. \textit{Suppose that $w_1> d_2$. If the transverse slice $\gamma$ of $f$ is $\mu_{m,2}$-minimal, then:}\\

        \noindent\textit{(a) the characteristic exponents of $\gamma$ are $m$ and $m+1$. Furthermore, $w_1=(n-1)w_2$ and $n>m+1$.}

        \noindent\textit{(b) we have that $p(x,y)=0$ and $xq(x,y)=c_1x^{w_2}y^{n-w_1}+\dots+c_{\theta}xy,$ for some $c_\theta\ne0.$}
    \end{lemma}

    \begin{proof}
        (a) By \cite[Prop. 3.10]{slice}, the characteristic exponents of the transverse slice $\gamma$ of $f$ are $m$ and $[(d_3-w_2)m+w_1]/w_1.$ But $\gamma$ is $\mu_{m,2}$-minimal. Then the characteristic exponents of $\gamma$ are $m$ and $m+1.$ Furthermore, we can have that $m+1=[(d_3-w_2)m+w_1]/w_1.$ Thus, $d_3-w_2=w_1.$ Since $d_3=nw_2,$ we obtain $w_1=(n-1)w_2.$ But $w_1>d_2=mw_2.$ Therefore, $n>m+1.$

        \noindent(b) By \cite[Lemma 3.8]{slice}, we have $p(x,y)=0$ and $xq(x,y)=c_1x^{w_1}y^{n-w_2}+\dots+c_{\theta}x^{\theta w_2}y$ with $c_\theta\ne0$ and $\theta=(n-1)/w_1.$ Since $w_1=(n-1)w_2,$ we conclude that $\theta w_2=1.$
    \end{proof}

    \begin{example}
        Let $m\ge2$ be a fixed integer. Consider the map germ $f(x,y)=(x,y^m,y^{m+2}+xy).$ We have that $f$ is a finitely determined, quasihomogeneous, corank 1 map germ of type $(m+1,m,m+2;m+1,1).$ Note that $f$ satisfies the assumptions in Lemma \rm{\ref{slice2}}. \textit{Therefore, the transverse slice of $f$ is $\mu_{m,2}$-minimal for all $m.$}
    \end{example}

    \begin{remark}
        Due to the assumption $w_1>d_2$ in Lemma \rm{\ref{slice2}}\textit{, no examples of finitely determined, homogeneous, corank 1 map germ with $\mu_{m,2}$-minimal transverse slice can be constructed.}
    \end{remark}

    \begin{lemma}\label{slice3}
        Let $f:(\mathbb{C}^2,0)\rightarrow(\mathbb{C}^3,0)$ be a finitely determined, quasihomogeneous, corank 1 map germ. Write $f$ in quasihomogeneous form $$f(x,y)=(x,y^m+xp(x,y),y^n+xq(x,y))$$ of type $(w_1,d_2,d_3;w_1,w_2)$ with $d_2\le d_3,$ as in \rm{\cite[\textit{Lemma} 2.11]{normalform}}. \textit{Suppose that $w_1\le d_2$ and $gcd(m,n)=2.$ If the transverse slice $\gamma$ of $f$ is $\mu_{m,3}$-minimal, then:}\\

        \noindent\textit{(a) the characteristic exponents of $\gamma$ are $w_1+1,w_1+3$ and $w_1+4,$ with $w_1$ odd, and $gcd(w_1,3)=1.$}

        \noindent\textit{(b) the weights of $f$ are distinct.}

        \noindent\textit{(c) $f$ is of the form $f(x,y)=(x,y^m+\alpha xy,y^{m+2}+\beta xy^3),$ for some $\alpha,\beta\in\mathbb{C},$ not both zero.}
    \end{lemma}

    \begin{proof}
        (a) By \cite[Prop. 3.10]{slice}, the characteristic exponents of the transverse slice $\gamma$ of $f$ are $d_2,d_3$ and $d_2+d_3-w_1.$ Since $gcd(m,n)=2$, we have that both $d_2$ and $d_3$ are even. Moreover, $V(x)$ is a component of $D(f).$ Hence, by \cite[Lemma 3.7]{slice}, it follows that $w_2=1$ and $w_1$ is odd. Now, since $\gamma$ is $\mu_{m,3}$-minimal and $m$ is even, the characteristic exponents of $\gamma$ are $m, m+2$ and $m+3.$ Therefore, $d_2=m, d_3=m+2$ and $d_2+d_3-w_1=m+3.$ From the latter equality, we obtain $w_1=m-1.$ Thus, the characteristic exponents of $\gamma$ are $w_1+1,w_1+3$ and $w_1+4,$ with $w_1$ odd. Now, since $gcd(d_1,d_3)=1,$ we conclude that $gcd(w_1,w_1+3)=1,$ and hence $gcd(w_1,3)=1.$

        \noindent(b) Since $\gamma$ has three characteristic exponents, it follows that $m\ge3.$ Therefore, as $w_1=m-1,$ we have $w_1\ge2.$ Hence, $w_1\ne w_2.$

        \noindent(c) We know that $m\ge3.$ Thus, since $w_1=m-1$ and $w_2=1,$ observe that $y^m+xp(x,y)=y^m+\alpha xy$ and $y^n+xq(x,y)=y^{m+2}+\beta xy^3,$ for some $\alpha,\beta\in\mathbb{C}.$ If $\alpha=\beta=0,$ then the number of cross-caps of $f$, given by the codimension of the ramification ideal of $f$ in $\mathcal{O}_2,$ would be infinite, which contradicts the finite determinacy of $f.$ Therefore, $\alpha$ and $\beta$ cannot both be zero.
    \end{proof}

    \begin{remark}\label{ex3}
        In the Introduction, we presented the family $f_t(x,y)=(x,y^4,x^5y+xy^5+y^6+ty^7),$ one of the counterexamples to Ruas' conjecture, provided in \rm{\cite{Ruas}}. \textit{For $t=0,$ the transverse slice $\gamma$ of $f$ has three characteristic exponents, namely $4,6$ and $9$, and thus $\gamma$ in not $\mu_{4,3}$-minimal. In fact, in Lemma} \rm{\ref{slice3}}\textit{, we show that there do not exists finitely determined, homogeneous, corank 1 map germs $f:(\mathbb{C}^2,0)\rightarrow(\mathbb{C}^3,0)$ whose slice is $\mu_{m,3}$-minimal. Observe that although $f_t$ has a $\mu_{4,3}$-minimal transverse slice for $t\ne0,$ $f_t$ is not homogeneous, which does not contradict Lemma} \rm{\ref{slice3}}.
    \end{remark}

    \begin{example}
    By Lemma \rm{\ref{slice3}}\textit{, we know that all finitely determined, quasihomogeneous, corank 1 map germs from $(\mathbb{C}^2,0)$ to $(\mathbb{C}^3,0)$ with $\mu_{m,3}$-minimal transverse slice $\gamma$ are of the form $$f(x,y)=(x,y^m+\alpha xy,y^{m+2}+\beta xy^3),$$ for some $\alpha,\beta\in\mathbb{C}$ not both zero. For example, consider the corank 1 map germ $g:(\mathbb{C}^2,0)\rightarrow(\mathbb{C}^3,0)$ given by $$g(x,y)=(x,y^6+xy,y^8+2xy^3).$$ This is a quasihomogeneous map germ of type $(5,6,8;5,1),$ and one can verify that $g$ is finitely determined. Moreover, the transverse slice $\gamma$ of $g$ is $\mu_{6,3}$-minimal, with characteristic exponents $6,8$ and $9.$}
\end{example}

Before presenting a result that provides an infinite list of counterexamples for the implication $(W) \Rightarrow (Top)$ in Ruas’ conjecture, we first recall the relevant definitions of equisingularity for unfoldings. 

\begin{definition}\label{defwe} 
    Let $f:(\mathbb{C}^n,0)\rightarrow(\mathbb{C}^{n+1},0)$ be a finitely determined map germ, and let $F:(\mathbb{C}^n\times\mathbb{C},0)\rightarrow(\mathbb{C}^{n+1}\times\mathbb{C},0)$ a 1-parameter unfolding of $f$ given by $F(x,t)=(f_t(x),t)$ and $I(x,t)=(f(x),t)$ the trivial unfolding of $f.$ We say that $F$ is topologically equisingular or topologically trivial if there exist germs of homeomorphisms $$G:(\mathbb{C}^n\times\mathbb{C},0)\rightarrow(\mathbb{C}^n\times\mathbb{C},0)\, \ \text{and}\, \ H:(\mathbb{C}^{n+1}\times\mathbb{C},0)\rightarrow(\mathbb{C}^{n+1}\times\mathbb{C},0)$$ such that $G$ and $H$ are 1-parameter unfoldings of identity map germs from $n$-space and $n+1$-space, respectively, such that $I=H\circ F\circ G.$ Furthermore, if $G, G^{-1}, H$ and $H^{-1}$ are Lipschitz maps with respect to the outer metric, then $F$ is called bi-Lipschitz equisingular or bi-Lipschitz trivial.
\end{definition}

Let us recall the notion of Whitney equisingularity for 1-parameter unfoldings. Gaffney defined in \cite{Gaffney} a concept for a class of unfoldings called excellent unfoldings. These unfoldings possess a natural stratification whose stratum complementing the parameter space $T$ (which is a sufficiently small neighborhood of $0$ in $\mathbb{C}$ for a representative $F$) are stable both in the source and the target. In the source, the stratification is given by \begin{equation}\label{fonte}
        \left\{\mathbb{C}^2\times\mathbb{C}\setminus D(F), D(F)\setminus T, T\right\}
    \end{equation} while the stratification in the target is given by \begin{equation}\label{meta}
        \left\{\mathbb{C}^3\times\mathbb{C}\setminus F(\mathbb{C}^2\times\mathbb{C}),F(\mathbb{C}^2\times\mathbb{C})\setminus F(D(F)) ,F(D(F))\setminus T, T\right\}.
    \end{equation}

    \begin{definition}\label{whit}
    We say that $F$ is Whitney equisingular if the stratifications in the source and target described in \eqref{fonte} and \eqref{meta} are Whitney equisingular along $T.$
\end{definition}

    \begin{proposition}\label{contraexemplo}
        Let $f:(\mathbb{C}^2,0)\rightarrow(\mathbb{C}^3,0)$ be a finitely determined, quasihomogeneous, corank 1 map germ. Write $f$ in quasihomogeneous form $$f(x,y)=(x,y^m+xp(x,y),y^n+xq(x,y))$$ of type $(w_1,d_2,d_3;w_1,w_2)$ with $d_2\le d_3,$ as in \rm{\cite[\textit{Lemma} 2.11]{normalform}}. \textit{Suppose that the transverse slice has three characteristic exponents and $d_1<d_2-1.$ Then there exists a topologically trivial 1-parameter unfolding $F=(f_t,t)$ of $f$ such that $F$ is not Whitney equisingular.}
    \end{proposition}

    \begin{proof}
        By \cite[Prop. 3.10]{slice}, we know that $gcd(m,n)=2$ and $V(x)\subset D(f).$ Consider a 1-parameter unfolding $F=(f_t,t)$ given by $$f_t(x,y)=(x+ty^{m-2},y^{m}+xp(x,y),y^n+xq(x,y)).$$ By \cite[Lemma 3.7]{slice}, $w_1$ is odd and $w_2=1.$ Thus, $F$ is of non-negative degree and by \cite[Corollary 1]{Damon}, we conclude that $F$ is topologically trivial. Now, note that $V(x)\subset D(f_t).$ Furthermore, $f(0,y)=(0,y^m,y^n)$ while $f_t(0,y)=(ty^{w_1-2},y^m,y^n)$ for all sufficiently small $t\ne0,$ and we have $m(f_t(V(x)),0)<m(f(V(x)),0).$ One can verify that the multiplicity of the image of the another components of the $D(f_t)$ is constant. By additivity of the multiplicity, we have that $2m(f_t(D(f_t)),0)<2m(f(D(f)),0)$ for all sufficiently small $t\ne0.$ Now, by \cite[Lemma 5.2]{[11]} and \cite[Lemma 4.24]{[10]}, $\mu(\gamma_t,0)=2m(f_t(D(f_t)),0).$ Hence, $F$ is not Whitney equisingular.
    \end{proof}

    \begin{corollary}\label{corcontra}
        Let $f:(\mathbb{C}^2,0)\rightarrow(\mathbb{C}^3,0)$ be a finitely determined, homogeneous, corank 1 map germ. If the transverse slice of $f$ has three characteristic exponents, then there exists a topologically trivial 1-parameter unfolding $F=(f_t,t)$ of $f$ such that $F$ is not Whitney equisingular.
    \end{corollary}

    \begin{proof}
        Write $f$ in the normal form $$f(x,y)=(x,y^m+xp(x,y),y^n+xq(x,y))$$ of the type $(1,d_2,d_3;1,1)$ with $d_2\le d_3,$ as in \cite[Lemma 2.11]{normalform}. Since the transverse slice of $f$ has three characteristic exponents, then $d_2=m\ge4.$ Thus, by \cite[Prop. 3.10]{slice} we obtain that $d_1<d_2-1.$ Now we conclude the proof by Proposition \ref{contraexemplo}.
    \end{proof}\\

    In Table \ref{table2}, we provide new counterexamples to Ruas' conjecture in Proposition \ref{contraexemplo} setting. Furthermore, third column in Table \ref{table2} shows $\mu(\gamma_t,0)$ for $t\ne0.$

    \begin{table}[H]
\centering
{\def\arraystretch{2.0}\tabcolsep=22pt 

\begin{tabular}{ c || c ||  c  }

\hline 
\rowcolor{lightgray}
 \textbf{Family of map germs} & \textbf{$\mu(\gamma,0)$} & \textbf{$\mu(\gamma_t,0)$}  \\
			  
\hline

$f_t(x,y)=(x+ty^{14},y^{16}+xy^{15},y^{18}+xy^{17}+x^{17}y)$  & $270$ & $268$ \\ \hline
     
$f_t(x,y)=(x+ty^{14},y^{16}+xy^{13},y^{22}+xy^{19}+x^7y)$    & $328$   & $326$ \\
\hline

$f_t(x,y)=(x+ty^{14},y^{16}+xy^{11},y^{26}+xy^{21}+x^5y)$ & $386$ & $384$   \\

\hline

$f_t(x,y)=(x+ty^{14},y^{16}+xy^9,y^{22}+xy^{15}+x^3y)$ & $324$ & $322$   \\

\hline
\end{tabular}
}
\caption{New counterexamples to Ruas' conjecture}\label{table2}
\end{table}

To finish this section, in the sequel we present a proposition that guarantees that Whitney equisingularity does not imply bi-Lipschitz triviality, in general. 

\begin{proposition}\label{whitbil}
    Let $f:(\mathbb{C}^2,0)\rightarrow(\mathbb{C}^3,0)$ be the homogeneous map given by $f(x,y)=(x^{14},y^8,(x-y)(x-2y)(x-3y)).$ Consider $F=(f_t,t)$ the 1-parameter unfolding of $f$ such that $f_t(x,y)=(x^{14},y^8,(x-y)(x-2y)(x-3y)+ty^{12}).$ Then, $F$ is Whitney equisingular but $F$ is not bi-Lipschitz trivial.
\end{proposition}

\begin{proof}
    Using Software Singular \cite{singular}, one can verify that $f$ is finitely determined. Furthermore, since $F$ is of non-negative degree, by \cite[Cor. 1]{Damon} we conclude that $F$ is topologically trivial. Now, we need to proof that $\mu(\widetilde{\gamma}_t,0)$ and $m(f_t(D(f_t)),0)$ are constant along the parameter space, where $\widetilde{\gamma}_t$ is the pre-image of the transverse slice of $f_t$. Let $(\mathcal{X},\mathcal{Y},\mathcal{Z})$ be a coordinate system of $\mathbb{C}^3$ and $H=V(a\mathcal{X}+b\mathcal{Y}+c\mathcal{Z})\subset\mathbb{C}^3$ a generic plane of $f_t,$ for all $t$ sufficiently small. Note that $$\widetilde{\gamma}_t=f_t^{-1}(H)=V(ax^{14}+by^8+c((x-y)(x-2y)(x-3y)+ty^{12})).$$ Thus, $\mu(\widetilde{\gamma}_t,0)=4,$ for all sufficiently small $t.$ Furthermore, $D(f_t)$ has 313 components. In particular, $V(x-y), V(x-2y)$ and $V(x-3y)$ are fold components and others components are identification components, for all $t.$ Hence, we obtain that $m(f_t(D(f_t)),0)=477,$ for all sufficiently small $t.$ Since $\mu(\widetilde{\gamma}_t,0)$ and $m(f_t(D(f_t)),0)$ are constant along the parameter space, by \cite[Th. 5.3]{[11]}, we have that $F$ is Whitney equisingular. Now suppose that $F$ is bi-Lipschitz trivial. Then, the reduced curve family $F(D(F))\subset\mathbb{C}^3\times\mathbb{C}$ is bi-Lipschitz trivial. Moreover, for each $t,$ we have that $V(x-y)\subset D(f_t).$ A parametrization of $f_t(V(x-y))$ in $\mathbb{C}^3$ is given by $\varphi_t(u)=(u^{7},u^4,tu^{6}).$ Thus, we obtain the reduced curve family $$\Phi: \mathbb{C}^2\rightarrow\mathbb{C}^3\times\mathbb{C}, \ \ \Phi(t,u)=(u^7,u^4,tu^6,t).$$ Now, using \cite[Ex. 5.2]{Snoussi} and \cite[Cor. 3.6]{Giles} we conclude that this family can not be topologically trivial, which contradicts the assumption of bi-Lipschitz triviality. Hence, $F$ is not bi-Lipschtiz trivial.
\end{proof}\\

The next result provide a positive answer for the implication $(d)\Rightarrow(c)$ in Ruas' conjecture in corank 1 case.

\begin{theorem}\label{ab}
    Let $f:(\mathbb{C}^2,0)\rightarrow(\mathbb{C}^3,0)$ be a finitely determined, corank 1 map germ. Consider $F=(f_t,t)$ a 1-parameter unfolding of $f.$ If $F$ is bi-Lipschitz trivial, then $F$ is Whitney equisingular.
\end{theorem}

\begin{proof}
    Since $f$ has corank 1, by \cite[Lemma 5.2]{[11]} we have that $\mu(\gamma_t,0)=2m(f_t(D(f_t)),0),$ for all $t$ sufficiently small. Now, since $F$ is bi-Lipschitz trivial, the family of curves $F(D(F))$ is bi-Lipschitz trivial. Consider $p:(F(D(F)),0)\rightarrow\mathbb{C}^2\times\mathbb{C}$ a generic projection. Since $p$ is generic, we have $m(f_t(D(f_t)),0)=m(p(f_t(D(f_t))),0),$ for all $t$ sufficiently small. By the definition of bi-Lipschitz triviality, the family of plane curves $p(f_t(D(f_t)))$ is topologically trivial. Thus, this family is equimultiple. Therefore, $m(f_t(D(f_t)),0)$ is constant along the parameter space. Now, the result follows by \cite[Th. 5.3]{[11]}.
\end{proof}

Note that the Proposition \ref{whitbil} provide a counterexample for the implication $(c)\Rightarrow(d)$ in Ruas' conjecture in corank 2 case. For other hand, Theorem \ref{ab} guarantees that the implication $(d)\Rightarrow(c)$ in Ruas' conjecture holds in corank 1 case. Thus, there exists two cases to consider for the relation of Whitney equisingularity and bi-Lipschitz trivilaty in Ruas' conjecture: the corank 1 case in the implication $(c)\Rightarrow(d)$ and the corank 2 case in the implication $(d)\Rightarrow(c).$

\section{Zariski's multiplicity conjecture for quasihomogeneous map germs}\label{sec4}

$ \ \ \ \ $ In the sequel, we turn our attention to the study of the 1-parameter unfoldings of finitely determined, quasihomogeneous, corank 1 map germs from $(\mathbb{C}^2,0)$ to $(\mathbb{C}^3,0).$ We now aim to present an affirmative solution to Question 2 in our setting, as outlined at the outset of this section. Furthermore, we establish the validity of Zariski’s multiplicity conjecture for this class of map germs.

    \begin{lemma}\label{Teo2}
        Let $f:(\mathbb{C}^2,0)\rightarrow(\mathbb{C}^3,0)$ be a finitely determined, quasihomogeneous, corank 1 map germ. If  $F=(f_t,t)$  is a topologically trivial 1-parameter unfolding of $f=(f_1,f_2,f_3)$, then $F$ is of non-negative degree, that is, any additional term $\alpha$ in the deformation of $f_i$ has weighted degree not smaller than that of $f_i$.
    \end{lemma}

\begin{proof}
    Since $f$ has corank 1, we can write it in the normal form given by $f(x,y)=(x,p(x,y),q(x,y))$ for some $p,q \in \mathfrak{m}^2$, where $\mathfrak{m}$ denotes the maximal ideal of $\mathcal{O}_2$ (see \cite[Lemma 4.1]{mond85}). Clearly, the unfolding $F=(f_t,t)$ has corank 1. Therefore, in this case we can write 
\begin{equation*}
f_t(x,y)=(x,p(x,y)+\widetilde{p}(x,y,t),q(x,y)+\widetilde{q}(x,y,t)).
\end{equation*}  
    Now, since $F$ is topologically trivial, the family of curves $D(f_t)$ is topologically trivial as well (\cite[Th. 6.2]{Bedregal}, see also \cite[Cor. 40]{lev}). Thus, since $f$ is quasihomogeneous, it follows that $D(f)$ is also quasihomogeneous (see for instance \cite[Prop. 1.15]{mond91}). Therefore, we can apply Varchenko’s result \cite{Varchenko} (see also \cite[Prop. 2]{OShea}) to conclude that $D(f_t)$ is a non-negative degree deformation of $D(f)$.

    On the other hand, by Remark \ref{resultant} we obtain that $D(f)=V(Res_{\phi,\psi,z}(x,y)),$ where $$\phi(x,y,z)=\dfrac{p(x,y)-p(x,z)}{y-z}\, \ \text{and}\, \ \psi(x,y,z)=\dfrac{q(x,y)-q(x,z)}{y-z}.$$ It follows by \cite[Prop. 5, p. 164]{ideal} that there exists unique $\widetilde{\phi},\widetilde{\psi}\in\mathbb{C}\left\{x,y,z\right\}$ such that    
\begin{equation*}
Res_{\phi,\psi,z}=\widetilde{\psi}\cdot\phi+\widetilde{\phi}\cdot\psi.
\end{equation*}      
    Using the fact that $f$ is finitely determined, we can regard $\phi, \tilde{\phi}, \psi$ and $\tilde{\psi}$ as polynomials in $z$. It follows that the degree in $z$ of $\widetilde{\phi}$ (respectively, $\widetilde{\psi}$) is smaller than the degree of $\phi$ (respectively, $\psi$). Since $deg(\phi),deg(\psi)>0$ and $Res_{\phi,\psi,z}\neq 0$ we obtain that both $\widetilde{\phi}$ and $\widetilde{\psi}$ are nonzero.
   
    For $t\ne0,$ let $\Phi_t$ and $\Psi_t$ be the generators of $D^2(f_t),$ that is, $$\Phi_t:=\dfrac{p(x,y)-\widetilde{p}(x,y,t)-p(x,z)+\widetilde{p}(x,z,t)}{y-z}=\phi+\phi_t$$ and similarly, 
    
    $$\Psi_t:=\dfrac{q(x,y)-\widetilde{q}(x,y,t)-q(x,z)+\widetilde{q}(x,z,t)}{y-z}=\psi+\psi_t.$$ 
    
    Since $D(f_t)=V(Res_{\Phi_t,\Psi_t,z}),$ there exists $h(x,y,t)$ such that $Res_{\Phi_t,\Psi_t,z}=Res_{\phi,\psi,z}+h(x,y,t).$ On the other hand, let $\widetilde{\phi}_t$ and $\widetilde{\psi}_t$ be such that $Res_{\Phi_t,\Psi_t,z}=\widetilde{\psi}_t\cdot\Phi_t+\widetilde{\phi}_t\cdot\Psi_t$ as before. Writing $$\widetilde{\phi}_t=\sum_{j\ge0}A_{j,t}(x,y,z)t^j\ \, \text{and}\ \, \widetilde{\psi}_t=\sum_{j\ge0}B_{j,t}(x,y,z)t^j,$$ we obtain $$\widetilde{\psi}\cdot\phi+\widetilde{\phi}\cdot\psi+h(x,y,t)=B_{0,t}\cdot\phi+A_{0,t}\cdot\psi+B_{0,t}\cdot\phi_t+A_{0,t}\cdot\psi_t+\sum_{j\ge1}B_{j,t}t^j\cdot\Phi_t+\sum_{j\ge1}A_{j,t}t^j\cdot\Psi_t,$$ where $\widetilde{\psi}\cdot\phi+\widetilde{\phi}\cdot\psi=B_{0,t}\cdot\phi+A_{0,t}\cdot\psi.$ By the uniqueness of $\widetilde{\phi}$ and $\widetilde{\psi},$ we conclude that $A_{0,t}=\widetilde{\phi}$ and $B_{0,t}=\widetilde{\psi}.$ Moreover, $$h(x,y,t)=\widetilde{\psi}\cdot\phi_t+\widetilde{\phi}\cdot\psi_t+\sum_{j\ge1}B_{j,t}t^j\cdot\Phi_t+\sum_{j\ge1}A_{j,t}t^j\cdot\Psi_t.$$ From the above equality, it follows that if $\widetilde{p}$ adds terms of weighted degree lower than those of $p$ or $\widetilde{q}$ adds terms of weighted degree lower than those of $q,$ then $h(x,y,t)$ will add terms of weighted degree lower than those of $Res_{\phi,\psi,z},$ which is a contradiction. Therefore, $F$ is of non-negative degree.
\end{proof}\\

Now, we obtain the validity of Zariski's multiplicity conjecture for the case of finitely determined, quasihomogeneous, corank 1 map germs from $(\mathbb{C}^2,0)$ to $(\mathbb{C}^3,0)$.

\begin{theorem}\label{zariski}
    Let $f:(\mathbb{C}^2,0)\rightarrow(\mathbb{C}^3,0)$ be a finitely determined, quasihomogeneous, corank 1 map germ. If $F=(f_t,t)$ is a topologically trivial 1-parameter unfolding of $f$, then $F$ is equimultiple.
\end{theorem}

\begin{proof}
    By Theorem \ref{Teo2}, we have that $F$ is of non-negative degree. Now, by \cite[Th. 5.3]{upper}, we conclude that $F$ is equimultiple.
\end{proof}

\begin{remark} The difficulty in extending Theorem \rm{\ref{zariski}} \textit{to finitely determined, quasihomogeneous, corank 1 map germs from $(\mathbb{C}^n,0)$ to $(\mathbb{C}^{n+1},0)$ using the same techniques arises from the fact that $D(f_t)$ does not have isolated singularity when $n\ge3,$ which prevents the application of Varchenko’s results} \rm{\cite{Varchenko}}. \textit{We leave this problem open for further investigation, considering the problem for $n\ge2$ and for any corank.}
\end{remark}

\begin{mybox}
		\textbf{Problem 1:} Let $f:(\mathbb{C}^n,0)\rightarrow(\mathbb{C}^{n+1},0)$ be a finitely determined quasihomogeneous map germ. If $F=(f_t,t)$ is a topologically trivial 1-parameter deformation of $f$, is it true that $F$ is of non-negative degree?
	\end{mybox}
	
\section{Ruas' conjecture for $\mu_{m,k}$-minimal transverse slices}\label{sec5}
	
	$ \ \ \ \ $  To conclude this work, we now present an answer to Question 1. However, before stating our main result, let us first consider a lemma that will assist in solving one of the cases in Theorem \ref{rua}. In this lemma, we make use of the normal form (in the analytic sense) obtained by the first author \cite[Lemma 2.11]{normalform} for finitely determined, quasihomogeneous, corank 1 map germs from $(\mathbb{C}^2,0)$ to $(\mathbb{C}^3,0)$. 

As with topological triviality, studying the Whitney equisingularity of an unfolding directly from the definition is not a simple task. However, Marar, Nuño-Ballesteros, and Peñafort-Sanchis \cite[Th. 5.3]{[11]} showed that the 1-parameter unfolding $F=(f_t,t)$ of a finitely determined map germ from $(\mathbb{C}^2,0)$ to $(\mathbb{C}^3,0)$  is Whitney equisingular if and only if $\mu(D(f_t),0)$ and $\mu(\gamma_t,0)$ are constant, where $\gamma_t$ is the transverse slice of $f_t.$ For convenience of the reader, we write Theorem \ref{1.1} again here, preserving the numeration of the section.

    \begin{theorem}\label{rua}
        Let $f:(\mathbb{C}^2,0)\rightarrow(\mathbb{C}^3,0)$ be a finitely determined, quasihomogeneous, corank 1 map germ. Suppose that $F=(f_t,t)$ is a topologically trivial 1-parameter unfolding of $f.$ If the transverse slice $\gamma$ of $f$ is $\mu_{m,k}$-minimal, then $F$ is Whitney equisingular.
    \end{theorem}

    \begin{proof}
        Since $f$ is singular, it follows that $m(\gamma,0)=m(f(\mathbb{C}^2),0)\ge2.$ Moreover, as $f$ is finitely determined, quasihomogeneous, corank 1, the first author \cite[Th. 1.1]{slice} showed that the slice $\gamma$ has only two or three characteristic exponents. Now, for $F=(f_t,t)$ a topologically trivial 1-parameter unfolding of $f$, Theorem \ref{zariski} guarantees that $F$ is equimultiple, that is, $m(\gamma_t,0)=m(\gamma,0),$ for all sufficiently small $t.$ Then, by Corollary \ref{upper}, we have $ce(\gamma_t,0)\le ce(\gamma,0)$ for sufficiently small $t,$ where $\gamma_t$ denotes the transverse slice of $f_t.$ Now, let us consider the possible cases:\\

        \noindent \underline{Case 1}: Suppose that $ce(\gamma,0)=2$ but $F$ is not Whitney equisingular. Then we must have $ce(\gamma_t,0)=1$ for all sufficiently small $t\ne0,$ which implies $m(\gamma_t,0)=1,$ a contradiction since $m(\gamma,0)\ge2.$\\

        \noindent \underline{Case 2}: Suppose that $ce(\gamma,0)=3.$ Thus, $\gamma$ is $\mu_{m,3}$-minimal and it follows from Lemma \ref{slice3} that $f$ is of the quasihomogeneous form $f(x,y)=(x,y^m+\alpha xy,y^{m+2}+\beta xy^3)$ of type $(w_1,w_1+1,w_1+3;w_1,1).$ Since $F$  is topologically trivial, Lemma \ref{Teo2} implies that $F$ is of non-negative degree. Thus, this deformation does not change the characteristic exponents of $\gamma.$ Therefore, $\mu(\gamma_t,0)=\mu(\gamma,0)$ for all sufficiently small $t\ne0,$ and by \cite[Th. 5.3]{[11]} it follows that $F$ is Whitney equisingular.\\

        Therefore, from Case 1 and Case 2, we conclude that $F$ must be Whitney equisingular.
\end{proof}

\begin{remark}\label{5.5}
(a) The corank 1 assumption in Theorem \rm{\ref{rua}} \textit{is essential, as illustrated by the following quasihomogeneous map germ $f:(\mathbb{C}^2,0)\rightarrow(\mathbb{C}^3,0)$ given by $$f(x,y)=(x^2,y^3+x^3,x^4+x^3y+x^2y^2+xy^3+y^4).$$ One can verify that $f$ is finitely determined and that its transverse slice is $\mu_{6,3}$-minimal; a normal form is given by $\varphi(u)=(u^6,u^8+u^9)$. Now, consider the 1-parameter unfolding $F=(f_t,t)$ given by $$f_t(x,y)=(x^2-ty^2,y^3+x^3,x^4+x^3y+x^2y^2+xy^3+y^4).$$ This deformation is of non-negative degree. Therefore, by Damon’s result} \rm{\cite[\textit{Cor.} 1]{Damon}}\textit{, $F$ is topologically trivial. However, we have $\mu(\gamma,0)=36$ while $\mu(\gamma_t,0)=35,$ for small $t\ne0.$ Thus, by} \rm{\cite[\textit{Lemma} 5.2]{[11]}}\textit{, we conclude that $F$ is not Whitney equisingular. Hence, the assumption that the map $f$ is of corank 1 cannot be dropped from Theorem} \rm{\ref{rua}}.

\noindent\textit{(b) As for the difficulty in the case where $f$ is not quasihomogeneous, it arises from the impossibility of using, once again, the results of Varchenko} \rm{\cite{Varchenko}}\textit{, which leads us to seek new techniques to address this case. Here, we leave the following question:}
\end{remark}

\begin{mybox}
		\textbf{Problem 2:} Let $f:(\mathbb{C}^2,0)\rightarrow(\mathbb{C}^3,0)$ be a finitely determined, corank 1 map germ, and let $F=(f_t,t)$ be a topologically trivial 1-parameter deformation of $f$. If the transverse slice $\gamma$ of $f$ is $\mu_{m,k}$-minimal, it is true that $F$ is Whitney equisingular?
	\end{mybox}

\section{A new formulation of Ruas’ conjecture}\label{sec8}

$ \ \ \ \ $ In this section we will study 1-parameter Whitney equisingular unfoldings of a map germ $f$ from $(\mathbb{C}^2,0)$ to $(\mathbb{C}^3,0)$. Along this section, $\widetilde{\gamma}$ will denote the pre-image of the transverse slice of $f.$ To begin with, we introduce our new invariant, denoted by $\mu((W(f),0))$, which is defined as the Milnor number (at $0$) of the plane curve germ $(W(f),0)$ defined by

\begin{center}
$(W(f),0):=(D(f) \cup \tilde{\gamma}(f), \ 0)$.
\end{center} 

The following lemma (item (c)) shows how the Milnor number of $(W(f),0)$ is related with other invariants of $f$.

\begin{lemma}\label{lemmaaux4} Let $f:(\mathbb{C}^2,0)\rightarrow (\mathbb{C}^3,0)$ be a finitely determined map germ. Then\\

\noindent (a) $i(D(f),\tilde{\gamma}(f))=2 m(f(D(f),0)$.

\noindent (b) $m(D(f),0)\cdot m(\tilde{\gamma},0) \leq 2m(f(D(f)),0)$. In particular, if $f$ has corank $1$, then we have 

\begin{center}
$m(D(f),0) \leq 2m(f(D(f)),0)$.
\end{center}

\noindent (c) $\mu(W(f),0)=\mu(D(f),0)+\mu(\tilde{\gamma},0)+4m(f(D(f),0)-1$.

\end{lemma}

\begin{proof} (a) Let $(\mathcal{X},\mathcal{Y},\mathcal{Z})$ be a system of coordinates of $\mathbb{C}^3$. After a change of coordinates, we can suppose that the plane $H=V(\mathcal{X})$ is generic for $f=(f_1,f_2,f_3)$ By the genericity of $H$, we can also suppose that 

\begin{equation}\label{eq11}
C(f(D(f))) \cap H= \lbrace 0 \rbrace.
\end{equation}

\noindent where $C(f(D(f)))$ denotes the Zariski tangent cone of $f(D(f))$. This means that if $\overrightarrow{v}=(v_1,v_2,v_3)$ is a vector with same direction as the tangent line of an irreducible component of $f(D(f))$, then $v_1 \neq 0$. Let $D(f)^i$ be an irreducible component of $D(f)$ and consider its image $f(D(f)^i)$ by $f$. Consider a  parametrization $\varphi_i:U\rightarrow D(f)^i$ of $D(f)^i$ defined by

\begin{center}
$\varphi_{i}(u)=(\varphi_{i,1}(u),\varphi_{i,2}(u))$,
\end{center}

\noindent where $U$ is a neighboorhood of $0$ in $\mathbb{C}$ and $\varphi_{i,1}(u), \varphi_{i,1}(u)\in \mathbb{C}\lbrace u \rbrace$, and suppose that it is primitive, i.e., $\varphi_{i}$ is generically $1$-to-$1$. Consider the mapping $\hat{\varphi}_{i}=f \circ \varphi_{i}(u): U \rightarrow f(D(f)^i)$, defined by  

\begin{equation}\label{eq12}
\hat{\varphi}_{i}(u):=(f_1(\varphi_i(u)),f_2(\varphi_i(u)),f_3(\varphi_i(u))).
\end{equation}

Suppose that $D(f)^i$ is an identification component of $D(f)$. Since the restriction of $f$ to $D(f)^i$ is generically $1$-to-$1$, the mapping $\hat{\varphi}_{i}$ in (\ref{eq12}) is a parametrization of $f(D(f)^i)$ (a primitive one). On the other hand, if $D(f)^i$ is a fold component of $D(f)$ then the restriction of $f$ to $D(f)^i$ is $2$-to-$1$, hence the mapping $\hat{\varphi}_{i}$ in (\ref{eq12}) is generically $2$-to-$1$ and it is also a parametrization of $f(D(f)^i)$ (a double cover one). In the sequel, $ord_u(p(u))$ denotes the degree of the non-zero term of lowest degree of $p(u) \in \mathbb{C}\lbrace u \rbrace$. If $D(f)^i$ is an identification component of $D(f)$, then the condition (\ref{eq11}) implies that $ord_u(f_1(\varphi_i(u)))=m(f(D(f)^i),0)$. On the other hand, if $D(f)^i$ is a fold component one then $ord_u(f_1(\varphi_i(u)))=2m(f(D(f)^i),0)$. Hence we have that 

\begin{center}
$i(D(f)^i,\tilde{\gamma})=ord_u f_1(\varphi_{i,1}(u),\varphi_{i,2}(u))=ord_u f_1(\varphi_i(u))$,
\end{center}

\noindent see for instance \cite[p. 174]{Greuel}. Denote the irreducible components of $D(f)$ by $D(f)^1,\cdots, D(f)^r$. Note that if $(D(f)^i,D(f)^j)$ is a pair of identification components then clearly $m(f(D(f)^i),0)=m(f(D(f)^j),0)$. Denote by $r_i(f)$ the number of identification components of $D(f)$ and $r_f(f)$ the number of fold components of $D(f).$ Suppose there are $r_i(f)=2k$ identification components of $D(f)$ and let $(D(f)^{i_1},D(f)^{j_1})$, $\cdots$, $(D(f)^{i_{k}},D(f)^{j_{k}})$ be all pairs of identifications components of $D(f)$. Let $D(f)^{l_1},\cdots,D(f)^{l_{r_f}}$ be all fold components of $D(f)$. Therefore $r=r_i(f)+r_f(f)$. By the additive property of the intersection multiplicity for plane curves and the additive property of the multiplicity for curves in general, we have that:

\begin{eqnarray*}
    i(D(f), \tilde{\gamma}(f))&=& \displaystyle \sum_{q=1}^{r} i(D(f)^{q},\tilde{\gamma}(f))\\
    &=&\left( \displaystyle \sum_{q=1}^{k} m(f(D(f)^{i_q}),0) \right)+\left( \displaystyle \sum_{q=1}^{k}  m(f(D(f)^{j_q}),0) \right)+ \left( \displaystyle \sum_{q=1}^{r_f(f)} 2 \cdot m(f(D(f)^{l_q}),0) \right)\\
    &=&\left( \displaystyle \sum_{q=1}^{k} 2 \cdot m(f(D(f)^{i_q}),0) \right)+ \left( \displaystyle \sum_{q=1}^{r_f(f)} 2 \cdot m(f(D(f)^{l_q}),0) \right)=2m(f(D(f)),0).
\end{eqnarray*}

\noindent (b) It follows by (a) and the property that $i(C^1,C^2)\geq m(C^1,0)\cdot m(C^2,0)$. If $f$ has corank $1$ then $\tilde{\gamma}$ is smooth and hence it has multiplicity $1$.\\

\noindent (c) Consider the germ of curve $(W,0):=(W(f),0)=(D(f) \cup \tilde{\gamma},0)$. Since $H$ is generic, we have that $(D(f) \cap \tilde{\gamma},0)=\lbrace 0 \rbrace$, i.e., $(D(f),0)$ and $(\tilde{\gamma},0)$ do not have any irreducible component in common. It is convenient to give a different notation for branches in $(D(f),0)$ and $(\tilde{\gamma},0)$. Let $(\tilde{\gamma}^1,0), \cdots, (\tilde{\gamma}^p,0)$ and $(D(f)^1,0),\cdots, (D(f)^r,0)$ be the irreducible components of $(\tilde{\gamma},0)$ and $(D(f),0)$, respectively. In this sense, denote by $(W^1,0),\cdots,(W^{p+r},0)$ the irreducible components of $(W,0)$. By \cite[Cor. 1.2.3]{Buchweitz}, we have that

\begin{center}
$\mu(W,0)= \displaystyle \sum_{q=1}^{p+r} (\mu(W^q,0)-1) \ + 2 \left( \displaystyle \sum_{q<j} i(W^q,W^j) \right) +1$ .
\end{center}

Again, by the additive property of the intersection multiplicity for plane curves we have that:

\begin{eqnarray*}
    \mu(W,0)&=& \displaystyle \sum_{q=1}^r (\mu(D(f)^q,0)-1)+ \displaystyle \sum_{q=1}^p (\mu(\tilde{\gamma}^q,0)-1)\\
    & & + \displaystyle 2 \left(\sum_{1 \leq q<j \leq r} i(D(f)^q,D(f)^j)+\displaystyle \sum_{1 \leq q<j \leq p} i(\tilde{\gamma}^q,\tilde{\gamma}^j)+\displaystyle  \sum_{q=1}^p \sum_{j=1}^r i(\tilde{\gamma}^q,D(f)^j) \right)+1\\
    &=& \left( \displaystyle \sum_{q=1}^r (\mu(D(f)^q,0)-1)+ \displaystyle 2 \left(\sum_{1<q,j<r} i(D(f)^q,D(f)^j)\right)+1 \right)\\
    &&+ \left( \displaystyle \sum_{q=1}^s (\mu(\tilde{\gamma}^q,0)-1) + \displaystyle 2 \left( \sum_{1<q,j<r} i(\tilde{\gamma}^q,\tilde{\gamma}^j)\right) + 1 \right)+ \displaystyle 2 i(D(f), \tilde{\gamma}) -1\\
    &=& \mu(D(f),0)+\mu(\tilde{\gamma},0)+4m(f(D(f),0)-1
\end{eqnarray*}

\noindent where the last equality follows by (a).\end{proof}\\

Finally we return our attention to Ruas' conjecture revisited. As a consequence of Lemma \ref{lemmaaux4}, the following result provide a positive answer for that question for Whitney equisingularity. For convenience of the reader, we write Theorem \ref{main result 3} again here, preserving the numeration of the section

\begin{theorem}\label{w_t}
    Let $f:(\mathbb{C}^2,0)\rightarrow(\mathbb{C}^3,0)$ be a finitely determined map germ and let $F:(\mathbb{C}^2 \times \mathbb{C},0)\rightarrow (\mathbb{C}^3 \times \mathbb{C},0)$, $F=(f_t,t)$, be an unfolding of $f$. Set $(W(f_t),0):=(D(f_t)\cup \widetilde{\gamma}_t,0)$, where $\widetilde{\gamma}_t$ is the pre-image of the transverse slice of $f_t$. Then

\begin{center}
$F$ is Whitney equisingular $ \ \ \ $ $\Longleftrightarrow$ $ \ \ \ $ $\mu(W(f_t),0)$ is constant.
\end{center}
\end{theorem}

\begin{proof} By Lemma \ref{lemmaaux4}(c), we have that 

\begin{equation}\label{eq10}
\mu(W_t,0)=\mu(D(f_t),0)+4m(f_t(D(f_t)),0)+\mu(\tilde{\gamma}_t,0)-1
\end{equation}

\noindent Now the result follows by \cite[Th. 5.3]{[11]} and the upper semi-continuity of the invariants in (\ref{eq10}).\end{proof}\\

Now, to obtain a new result about Ruas' conjecture, we can study the existence of a family of plane curves whose constancy of the Milnor number of their fibers is necessary and sufficient to ensure the bi-Lipschitz equisingularity of a $1$-parameter unfolding of a finitely determined map germ from $(\mathbb{C}^2,0)$ to $(\mathbb{C}^3,0)$. In other words, we have the next problem

\begin{mybox}
		\textbf{Problem 3:} Let $f:(\mathbb{C}^2,0)\rightarrow(\mathbb{C}^3,0)$ be a finitely determined map germ. Consider $F=(f_t,t)$ a 1-parameter unfolding of $f.$ There exists a family $L_t$ of plane curves such that the bi-Lipschtiz triviality of $F$ is characterized by the constancy of $\mu(L_t,0)$?
	\end{mybox}

An interesting case is when $\tilde{\gamma}$ is transversal to $D(f)$. In this case, the intersection multiplicity $i(D(f),\tilde{\gamma})$ is simply the product of the multiplicities of $D(f)$ and $\tilde{\gamma}$. Under this hypothesis, if $F$ is a topological trivial 1-parameter unfolding of a corank $1$ map germ, then it is also Whitney equisingular. This case is explored in the following corollary.

\begin{corollary}\label{app1} Let $f:(\mathbb{C}^2,0)\rightarrow(\mathbb{C}^3,0)$ be a finitely determined map germ. Suppose now that $\tilde{\gamma}$ is transversal to $D(f)$, i.e, $\tilde{\gamma}$ and $D(f)$ do not have any tangent in common. Then\\

\noindent (a) $2m(f(D(f),0)= m(D(f),0)\cdot m(\tilde{\gamma},0)$.\\

In particular, if $f$ has corank $1$, then $m(D(f),0)=2m(f(D(f)),0)$.\\

\noindent (b) In addition, suppose that $f$ has corank $1$. Let $F:(\mathbb{C}^2 \times \mathbb{C},0)\rightarrow (\mathbb{C}^3 \times \mathbb{C},0)$, $F=(f_t,t)$, be an unfolding of $f$. If $F$ is topologically trivial, then it is Whitney equisingular.
\end{corollary}

\begin{proof} The proof of (a) follows by Lemma \ref{lemmaaux4}(a) and the fact that in this case we have that

\begin{center}
$i(D(f),\tilde{\gamma})=m(D(f),0) \cdot m(\tilde{\gamma},0)$. 
\end{center}

\noindent If $f$ has corank $1$ then $\tilde{\gamma}$ is smooth and hence $m(\tilde{\gamma},0)=1$. For the proof of (b), note that the assumption of transversality of $D(f)$ and $\tilde{\gamma}$ implies that $m(D(f),0)=2m(f(D(f)),0)$. Since $F$ is topologically trivial, then $D(f_t)$ is a topological trivial family of plane curves, in particular, $m(D(f_t),0)$ is constant. Note that $\tilde{\gamma}_t$ is smooth for all $t$, hence $m(\tilde{\gamma}_t,0)=1$ for all $t$. Thus we have that 

\begin{center}
$m(D(f),0)=i(D(f),\tilde{\gamma})\geq i(D(f_t),\tilde{\gamma}_t)\geq m(D(f_t),0)=m(D(f),0)$
\end{center}.  

\noindent Hence, $m(f_t(D(f_t)),0)$ is constant. Now the result follows by \cite[Cor. 8.9]{Gaffney}.\end{proof}\\

Another consequence of Lemma \ref{lemmaaux4} is about the multiplicity of double fold map germs from $(\mathbb{C}^2,0)$ to $(\mathbb{C}^3,0)$ which are map germs in the form $f(x,y)=(x^2,y^2,h(x,y))$. These kind of map germs was introduced by Marar and Nuño-Ballesteros in \cite{[3]}. The map germ $f:(\mathbb{C}^2,0)\rightarrow (\mathbb{C}^3,0)$ defined by

\begin{center}
$f(x,y)=(x^2,y^2,x^3+y^3+xy)$.
\end{center}

\noindent is a typical example of a double fold map germ and was considered in \cite{[11]}. In this example, we have that $f$ has corank $2$ and it is finitely determined. In \cite[Ex. 5.4]{[11]}, Marar, Nuño-Ballesteros and Peñafort-Sanchis showed that and $m(D(f),0)=m(f(D(f)),0)=5$. The following result shows that the equality of the multiplicities of these curves actually always occurs for map germs of this type.

\begin{corollary} If $f:(\mathbb{C}^2,0)\rightarrow (\mathbb{C}^3,0)$ is a corank $2$ finitely determined double fold map germ, i.e. a map germs of the form $f(x,y)=(x^2,y^2,h(x,y))$, then

\begin{center}
 $m(D(f),0)=m(f(D(f)),0)$.
\end{center}
 
\end{corollary} 

\begin{proof} Since $f$ has corank $2$ then $h(x,y) \in \mathfrak{m}_2^2$, where $\mathfrak{m}_2$ denotes the maximal ideal of $\mathcal{O}_2$. We have that for $\tilde{\gamma}=V(\alpha_1 x^2 + \alpha_2 y^2+ \alpha_3 h(x,y))$ where $\alpha_i$ are generic constants. Note that $\tilde{\gamma}$ has the topological type of a Morse singularity, i.e., locally $\tilde{\gamma}$ consist of two smooth branches with transversal intersection, in particular $m(\tilde{\gamma},0)=2$. The genericity hypothesis implies that the curves $\tilde{\gamma}$ and $D(f)$ are transversal. Hence, by Corollary \ref{app1} we conclude that  $2m(f(D(f)),0)=m(D(f),0)\cdot m(\tilde{\gamma},0)=2m(D(f),0)$.\end{proof}\\

We finish with the following corollary which can be viewed as a simple case of Zariski's multiplicity conjecture for families of parametrized surfaces.

\begin{corollary}\label{corZariskiparafamilias} Let $f:(\mathbb{C}^2,0)\rightarrow (\mathbb{C}^3,0)$ be a finitely determined, quasihomogeneous map germ. Consider an 1-parameter unfolding $F:(\mathbb{C}^2 \times \mathbb{C},0)\rightarrow (\mathbb{C}^3 \times \mathbb{C},0)$, $F=(f_t,t)$. If $F$ adds only terms of the same degrees as the degrees of $f$, that is, $f_t$ is quasihomogeneous of the same type for all $t$, then $m(f_t(\mathbb{C}^2),0)$ is constant along the parameter space.
\end{corollary}

\begin{proof} It is a directly consequence of \cite[Lemma 4.2]{OtoZari}.\end{proof}\\

Consider $f$ as in Corollary \ref{corZariskiparafamilias}, if $f$ has corank $1$ then we can say more about the unfolding $F$.  In \cite[Cor. 4.4]{slice} the first author shows that if in addition $f$ has corank $1$, then $F$ is Whitney equisingular. One may ask if under the conditions of the Corollary \ref{corZariskiparafamilias} (without corank $1$ hypothesis) we also should have that $F$ is Whitney equisingular. Unfortunately the answer is negative (see \cite[Ex. 5.4]{Ruas}).

    \begin{remark}
        The authors used the software Surfer \rm{\cite{surfer}} \textit{to create all figures in text}.
    \end{remark}

\section*{Acknowlegments}	
$ \ \ \ \ $ This work constitutes a part of the second author's Ph.D. thesis at the Federal University of Paraíba under the supervision of Otoniel Nogueira da Silva to whom he would like to express his deepest gratitude. Moreover, the first author acknowledges support by grant Universal 407454/2023-3 from Conselho Nacional de Desenvolvimento Científico e Tecnológico (CNPq). The second author acknowledges support by Coordenação de Aperfeiçoamento de Pessoal de Nível Superior (CAPES). Finally, we thank Roberto Giménez Conejero for constructive suggestions that improved this manuscript, and Saurabh Trivedi for directing us to Thom's problem.

	\begin{flushleft}
		$\bullet$ Silva, O. N.\\
		\textit{otoniel.silva@academico.ufpb.br}\\
		Universidade Federal da Paraíba, 58.051-900, João Pessoa, PB, Brazil.\\
		
		$ \ \ $\\
		
		$\bullet$ Silva Jr, M. M.\\
		\textit{mmsj@academico.ufpb.br}\\
		Universidade Federal da Paraíba, 58.051-900, João Pessoa, PB, Brazil.	
	\end{flushleft}

\begin{thebibliography}{3}
		\small
		\setlength{\itemsep}{0pt}      
\setlength{\parskip}{0pt}      
\setlength{\parsep}{0pt}       

		\bibitem{ACampo} A'Campo, N.: \textit{Le nombre de Lefschetz d'une monodromie}. Indag. Math. \textbf{35} (1973), 113--118.

        \bibitem{Apostol} Apostol, T. M.: \textit{Introduction to analytic number theory}. New York: Springer-Verlag, 1976.
               
        \bibitem{Briancon} Briançon, J.; Galligo, A.; Granger, M.: \textit{Deformations équisingulières des germes de courbes gauches réduites}. Mém. Soc. Math. Fr. (N.S.) \textbf{2} (1980), 1--69.

        \bibitem{Buchweitz} Buchweitz, R. O.; Greuel, G.-M.: \textit{The Milnor number and deformations of complex curve singularities}.  Invent. Math., \textbf{58} (1980), 241--281.
        
        \bibitem{Bedregal} Callejas-Bedregal, R.; Houston, K.; Ruas, M. A. S.: \textit{Topological Triviality of Families of Singular Surfaces}. arXiv:math/0611699 [math.CV] (2006).

        \bibitem{ideal} Cox, D. A.; Little, J.; O'Shea, D.: \textit{Ideals, Varieties, and Algorithms.} 3rd ed., Springer, New York, (2007).

        \bibitem{Damon} Damon, J.: \textit{Topological triviality and versality for subgroups of $A$ and $K. II$.} Sufficient conditions and applications. Nonlinearity, \textbf{5}(2) (1992), 373--412.

        \bibitem{Eyral} Eyral, C.: \textit{Zariski's multiplicity question - a survey}. N. Z. J. Math.  \textbf{36} (2007), 253--276.
		
		\bibitem{lev} Fernández de Bobadilla, J.; Pe Pereira, M.: \textit{Equisingularity at the normalization}. J. Topol. \textbf{1} (2008), no. 4, 879--909.

        \bibitem{Pelka} Fernández de Bobadilla, J.; Pe\l{}ka, T.: \textit{Symplectic monodromy at radius zero and equimultiplicity of $\mu$-constant families}. Ann. of Math. (2) \textbf{200} (2024), 153--299.

        \bibitem{Snoussi} Fernández de Bobadilla, J.; Snoussi, J.; Spivakovsky, M.: \textit{Equisingularity in One-Parameter Families of Generically Reduced Curves}. Int. Math. Res. Not. \textbf{5}, (2017) 1589-1609.

        \bibitem{Gaffney} Gaffney, T.: \textit{Polar multiplicities and equisingularity of map germs}. Topology, \textbf{32} (1993). 185--223. 

        \bibitem{Trot} Gaffney, T.; Trotman, D.; Wilson, L.: \textit{Equisingularity of sections, $(t^r)$ condition, and the integral closure of modules}. J. Algebraic Geom. \textbf{18}(4), (2009) 651--689.

        \bibitem{Gibson} Gibson, C. G.; Wirthmuller, K.; du Plessis, A. A.; Looijenga, E. J. N.: \textit{Topological stability of smooth mappings}. Lecture Notes in Math. 552, Springer, Berlin, 1976.

        \bibitem{Giles} Giles Flores, A.; Silva, O. N.; Snoussi, J.: \textit{On Tangency in Equisingular Families of Curves and Surfaces}. Q. J. Math. \textbf{71}(2), (2020) 485-505.
        
 \bibitem{roberto} Giménez Conejero, R.; Nuño-Ballesteros, J. J..: \textit{Singularities of mappings on ICIS and applications to Whitney equisingularity}. Adv. Math. \textbf{408} 2022.

        \bibitem{Greuel} Greuel, G. M.; Lossen, C.; Shustin, E.: \textit{Introduction to singularities and deformations}. Springer, Berlin, 2000.

        \bibitem{Abramo} Hefez, A.: \textit{Irreducible plane curve singularities}. Real and complex singularities. Eds D. Mond and M. J. Saia. Lecture Notes in Pure and Appl. Math. 232, New York, Marcel Dekker (2003).

        \bibitem{Hironaka} Hironaka, H.:  \textit{Normal cones in analytic Whitney stratifications}. Publ. Math. Inst. Hautes Études Sci. \textbf{36} (1969), 325--352.
        
        \bibitem{Houston} Houston, K.: \textit{Equisingularity of families of hypersurfaces and applications to mappings}. Michigan Math. J. \textbf{60}(2), (2011) 289–312.
		
		\bibitem{MararMond} Marar, W. L.; Mond, D.: \textit{Multiple points schemes for corank 1 maps}. J. London Math. Soc. \textbf{39} (1989), 553--567.

        \bibitem{[3]} Marar, W. L.; Nuño-Ballesteros, J. J.: \textit{A note on finite determinacy for corank 2 map germs from surfaces to 3-space.} Math. Proc. Camb. Philos. Soc. \textbf{145}, (2008) 153--163.
        
        \bibitem{MararJuan} Marar, W. L.; Nuño-Ballesteros, J. J.: \textit{Slicing corank 1 map germs from $\mathbb{C}^2$ to $\mathbb{C}^3$}. Quart. J. Math. \textbf{65} (2014), 1375--1395.
				
		\bibitem{[11]} Marar, W. L.; Nuño-Ballesteros, J. J.; Peñafort-Sanchis, G.: \textit{Double point curves for corank 2 map germs from $\mathbb{C}^2$ to $\mathbb{C}^3$.} Topology Appl. \textbf{159} (2012), 526--536.

        \bibitem{Milnor} Milnor, J.: \textit{Singular points of complex hypersurfaces} Princeton Univ. Press, Princeton, NJ, 1968.
        
        \bibitem{mond85} Mond, D.: \textit{On the classification of germs of maps from $\mathbb{R}^2$ to $\mathbb{R}^3$}. Proc. London Math. Soc. \textbf{50 (3)}, no. 2, 333--369, 1985.
        
        \bibitem{mond91} Mond, D.: \textit{The number of vanishing cycles for a quasihomogeneous mapping from $\mathbb{C}^2$ to $\mathbb{C}^3$}; Quart. J. Math. Oxford \textbf{42 (2)}, 335--345, 1991.
		
        \bibitem{[8]} Mond, D.: \textit{Some remarks on the geometry and classification of germs of maps from surfaces to 3-space.} Topology \textbf{26} (1987), 361--383.
	
		\bibitem{[7]} Mond, D.; Nuño-Ballesteros, J. J.: \textit{Singularities of Mappings}. volume 357 of Grundlehren der mathematischen Wissenschaften. Springer-Verlag, New York, first edition, (2020).
		
        \bibitem{Navarro} Navarro Aznar, V.: \textit{Sobre la invariància topològica de la multiplicitat}. Publ. Sec. Mat. Univ. Autònoma Barcelona \textbf{20} (1980), 191-192.
        
        \bibitem{Pichon} Neumann, W. D.; Pichon, A.: \textit{Lipschitz geometry of complex curves}. J. Singul. \textbf{10} (2014), 225--234.
        
    \bibitem{Nuno2} Nuño-Ballesteros, J. J.; Tomazella, J. N.:  \textit{Equisingularity of families of map germs between curves}. Math. Z. \textbf{272} (2012), 349--360.
    
 \bibitem{OShea} O'Shea, Donal B.: \textit{Topologically trivial deformations of isolated quasihomogeneous hypersurface singularities are equimultiple}. Proc. Amer. Math. Soc. \textbf{101} (1987), no. 2, 260--262.
		
        \bibitem{Rudd} Pellikaan, R.: \textit{Hypersurface singularities and resolutions of Jacobi modules}. Dissertation, Rijksuniversiteit te Utrecht, Utrecht, 1985.

        \bibitem{Conjruas} Ruas, M. A. S.: \textit{On the equisingularity of families of corank 1 generic germs}. Differential Topology, Foliations and Groups Action. Contemp. Math., \textbf{161} (1994), 113--121.

        \bibitem{Tentruas} Ruas, M. A. S.: \textit{Equimultiplicity of topologically equisingular families of parametrized surfaces in $\mathbb{C}^3$}. arXiv:1302.5800 [math.CV] (2013).

        \bibitem{Ruas} Ruas, M. A. S.; Silva, O. N.: \textit{Whitney equisingularity of families of surfaces in $\mathbb{C}^3$}. Math. Proc. Camb. Philos. Soc. \textbf{166}(2) (2019), 353--369.

        \bibitem{Saeki} Saeki, O.: \textit{Topological invariance of weights for weighted homogeneous isolated singularities in $\mathbb{C}^3$}. Proc. Amer. Math. Soc., \textbf{103}(3) (1988), 905--909.

        \bibitem{[10]} Silva, O. N.: \textit{Superfícies com singularidades não isoladas}. Tese (Doutorado em Matemática) - Instituto de Ciências Matemáticas e de Computação, Universidade de São Paulo, São Carlos, Brazil, 2017.

        \bibitem{upper} Silva, O. N.: \textit{Equimultiplicity of families of map germs from $\mathbb{C}^2$ to $\mathbb{C}^3$}. Bull. Braz. Math. Soc., New Series \textbf{51} (2020), 429--447.

        \bibitem{normalform} Silva, O. N.: \textit{On invariants of generic slices of weighted homogeneous corank 1 map germs from the plane to 3-space}. Bull. Braz. Math. Soc. \textbf{52} (2021), 663--677.

         \bibitem{slice} Silva, O. N.: \textit{On the topology of the transversal slice of a quasi-homogeneous map germ}. Math. Proc. Cambridge Philos. Soc. \textbf{176}(2) (2024), 339--359.

         \bibitem{OtoZari} Silva, O. N.: \textit{On Zariski multiplicity conjecture for quasihomogeneous surfaces with non-isolated singularities}. Preprint: arXiv:2309.02624 [math.AG] 20 Jul 2024.
         
 \bibitem{surfer}      \newblock \emph{Surfer}, Imaginary. \url{https://imaginary.org/program/surfer} (accessed 2025-10-07).


        \bibitem{Le} Tráng, L. D.: \textit{Calcul du nombre de cycles évanouissants d'une hypersurface complexe}. Ann. Inst. Fourier (Grenoble) \textbf{23} (1973), 261--270.
        
 \bibitem{Trotman} Trotman, D.: \textit{Stratification Theory}. In: Cisneros-Molina, J. L.; Tráng, L. D.; Seade, J. (Eds.), Handbook of Geometry and Topology of Singularities I, Springer, Cham, (2020), 243--273.


        \bibitem{Varchenko} Varchenko, A. N.: \textit{A lower bound for the codimension of the stratum $\mu=\text{const}$ in terms of the mixed Hodge structure}. Vestnik Moskov. Univ. Ser. I Mat. Mekh. \textbf{37} (1982), 28--31. (Russian)

        \bibitem{Wall} Wall, C. T. C.: \textit{Finite determinacy of smooth map germs}. Bull. London. Math. Soc. \textbf{13} (1981), 481--539.

        \bibitem{singular} Wolfram, D.; Greuel, G.-M.; Pfister, G.; Schönemann, H.: \textit{Singular 4-0-2. A computer algebra system for polynomial computations}. 2015. \url{https://www.singular.uni-kl.de}.

        \bibitem{whitney65} Whitney, H.:  \textit{Tangents to an analytic variety}. Ann. of Math. \textbf{81} (1965), 496--549.
		
		\bibitem{Xu} Xu, Y.; Yau, S. S.-T.: \textit{Classification of topological types of isolated quasi-homogeneous two dimensional hypersurfaces singularities}. Manuscripta Math. \textbf{64} (1989), 44--469.
		
		\bibitem{Yau} Yau, S. S.-T.: \textit{Topological types and multiplicities of isolated quasi-homogeneous hypersurfaces singularities}. Bull. Amer. Math. Soc. \textbf{19} (1988), 447--454.
        
     \bibitem{Zariski1979} Zariski, O.: \textit{Foundations of a general theory of equisingularity on $r$-dimensional algebroid and algebraic varieties, of embedding dimension $r+1$}. Amer. J. Math. \textbf{101} (1979), 453--514.

        \bibitem{Zariskii} Zariski, O.: \textit{On the topology of algebroids singularities}. Am. J. Math. \textbf{54} (1932), 453--465.

        \bibitem{Conjzari} Zariski, O.: \textit{Some open questions in the theory of singularities}. Bull. Amer. Math. Soc. \textbf{77} (1971), 481--491.
        
        \bibitem{Zariski} Zariski, O.: \textit{Collected papers}. MIT Press, Cambridge, MA, \textbf{4}, (1979).
        		
		
	\end{thebibliography}
\end{document}